\documentclass[11pt]{amsart}

\usepackage[colorlinks=true, pdfstartview=FitV, linkcolor=blue, 
citecolor=blue]{hyperref}
\usepackage[usenames,dvipsnames]{xcolor}

\usepackage{amssymb,amsmath,amscd}
\usepackage{graphicx}
\usepackage{bbm}
\usepackage{a4wide}
\usepackage{enumitem}
\usepackage{verbatim}
\usepackage{float}
\usepackage{accents}
\usepackage{stackrel}
\usepackage{faktor}
\usepackage{epsfig,psfrag}  
\usepackage{etex} 
\usepackage{pgfplots}
\usepackage{array}
\usepackage{mathrsfs}
\usepackage{amsthm}
\usepackage{mathtools}
\usepackage{tikz}

\makeatletter
\newcommand*{\boxwedge}{%
  \mathbin{%
    \mathpalette\@boxwedge{}%
  }%
}
\newcommand*{\@boxwedge}[2]{%
  \sbox0{$#1\boxplus\m@th$}%
  \dimen2=.5\dimexpr\wd0-\ht0-\dp0\relax 
  \dimen@=\dimexpr\ht0+\dp0\relax
  \def\lw{.06}
  \kern\dimen2 
  \tikz[
    line width=\lw\dimen@,
    line join=round,
    x=\dimen@,
    y=\dimen@,
  ]
  \draw
    (\lw/2,0) rectangle (1-\lw,1-\lw)
    (\lw,0) -- (.5,1-\lw-\lw/2) -- (1-\lw-\lw/2 ,0)
  ;%
  \kern\dimen2 
}
\makeatother

\theoremstyle{plain}
\newtheorem{theorem}{Theorem}[section]
\newtheorem{prop}[theorem]{Proposition}
\newtheorem{lemma}[theorem]{Lemma}
\newtheorem{coro}[theorem]{Corollary}

\newtheorem{assumption}{Assumption}

\theoremstyle{definition}
\newtheorem{remark}[theorem]{Remark}
\newtheorem{example}[theorem]{Example}
\newtheorem{definition}[theorem]{Definition}

\newcommand{\ts}{\hspace{0.5pt}}

\newcommand{\RR}{\mathbb{R}\ts}
\newcommand{\PP}{\mathbb{P}\ts}

\newcommand{\NN}{\mathbb{N}}

\newcommand{\EE}{\mathbb{E}}
\newcommand{\cA}{\mathcal{A}}
\newcommand{\cB}{\mathcal{B}}
\newcommand{\cC}{\mathcal{C}}
\newcommand{\cD}{\mathcal{D}}
\newcommand{\cE}{\mathcal{E}}
\newcommand{\cG}{\mathcal{G}}

\newcommand{\cL}{\mathcal{L}}

\newcommand{\cO}{\mathcal{O}}
\newcommand{\cP}{\mathcal{P}}
\newcommand{\cQ}{\mathcal{Q}}
\newcommand{\cR}{\mathcal{R}}

\newcommand{\one}{\mathbbm{1}}

\newcommand{\bP}{{\textbf{P}}}

\newcommand{\ee}{\mathrm{e}}
\newcommand{\dd}{\, \mathrm{d}}

\newcommand{\psimut}{{\psi_{\textnormal{mut}}^{}}}
\newcommand{\psisel}{{\psi_{\textnormal{sel}}^{}}}

\definecolor{gre}{rgb}{.06,.49,0.03} 

\DeclareMathOperator{\id}{Id}

\newcommand{\defeq}{\mathrel{\mathop:}=}
\newcommand{\eqdef}{=\mathrel{\mathop:}}

\tikzset{
    cross/.pic = {
    \draw[rotate = 45] (-#1,0) -- (#1,0);
    \draw[rotate = 45] (0,-#1) -- (0, #1);
    }
}

\begin{document}

\title[Recombination and more] {Labelled partitions in action: recombination, selection, mutation, and more}

\author{F. Alberti}
\address[F. Alberti]{Arbeitsgruppe Stochastik, Johannes Gutenberg University Mainz, \newline
\hspace*{\parindent}55122 Mainz, Germany}
\email{fralbert@uni-mainz.de}

\begin{abstract}
In this paper, we consider the evolution of an (infinitely large) population under recombination and additional evolutionary forces, modelled by a measure-valued ordinary differential equation. We provide a stochastic representation for the solution of this model via duality to a new labelled partitioning process with Markovian labels. In the special case of single-crossover, this leads to a recursive solution formula. This extends (and unifies) previous results on the selection-recombination equation. As a concrete example, we consider the selection-mutation-recombination equation. 
\end{abstract}

\maketitle


\bigskip


\section{Introduction}
Models that describe the effect of genetic recombination have long been among the big challenges for mathematical population geneticists. While the first statement of the deterministic recombination equation dates back over a century~\cite{Jennings,Robbins}, an explicit solution has long seemed out of reach. To tame the notorious non-linearity of this system of equations, a key idea, originating from the study of \emph{genetic algebras}~\cite{Lyubich,HaleRingwood}, was to embed the solution of the recombination equation into a higher dimensional space in which it solves a \emph{linear equation}. However, it was not until recently that this idea, known as \emph{Haldane linearisation}, was succesfully applied to obtain an explicit solution~\cite{haldane,BaakeBaakeSalamat}. The crucial new insight was to exploit the connection between the nonlinear evolution forward in time and a Markov process describing individual lines of descent \emph{backward} in time. The solution of the nonlinear recombination equation can be expressed in terms of the law of this process, which satisfies a linear equation. This way of thinking is strongly
related to the \emph{ancestral recombination graph}~\cite{hudson,griffithsmarjoram96,griffithsmarjoram97,JenkinsFearnheadSong,BhaskarSong,LambertSchertzer}, see also \cite[Ch.~3.4]{durrett}.

Having obtained an explicit solution of the recombination equation, the logical next step was to consider the interplay with other evolutionary forces. In~\cite{SelrekDocumenta}, we solved the \emph{selection-recombination equation} for a single selected site and single-crossover, by combining the aforementioned ancestral recombination graph with the \emph{ancestral selection graph}~\cite{KroneNeuhauser97} into the \emph{ancestral selection-recombination graph} (ASRG). Here, single-crossover means that, upon reproduction, the parental sequences cross over at exactly one sequence position, and the subsequences to either side of the crossover point are recombined in the offspring individual. This, together with the uniqueness of the selective site induces a natural order (see Section~\ref{sec:singlecrossover}) on the set of sequence sites which allowed us to distill the rather complicated ASRG into three distinct, simpler dual processes; namely, a weighted partitioning process, a family of Yule processes with initiation and resetting (YPIR), and a family of \emph{initiation processes}, each of them providing different insight into the genealogical structure of a sample. In the companion paper~\cite{SelrekTPB}, we mainly focused on the special case where the selected site was located at the boundary of the sequence. We introduced yet another genealogical construction which we called the \emph{ancestral initiation graph} (AIG), and considered concrete biological applications.

In this contribution, we take a more abstract perspective and consider a general model that describes the evolution under (not necessarily single-crossover) recombination and other evolutionary forces that satisfy certain natural assumptions; in essence, we will assume that their action depends only on a single `active' genetic site. Our motivation is twofold. Firstly, we want to generalise the results from~\cite{SelrekDocumenta,SelrekTPB} to, for instance, the case with selection, recombination \emph{and mutation}. Secondly, and perhaps more importantly, we hope to better bring out the mathematical structure of the approach in the aforementioned work, and obtain a unifying description of the `zoo of dualities' obtained there. Indeed, one of the main contributions is to introduce a \emph{general labelled partitioning process} (gLPP), which simultaneously generalises all the aforementioned dualities. To the best of our knowledge, this is the first time that such a general setting has been considered in the context of mathematical models for recombination. Evans~\cite{EvansBLA} also discusses a general labelled partitioning process, but in the context of genetic drift/resampling rather than recombination. Moreover, in~\cite{EvansBLA}, the partitions act on samples of individuals and not, as in this work, on the set of genetic sites.

This paper is organised as follows. In Section~\ref{sec:psirecoeq}, we introduce our model, namely the general recombination equation as considered in~\cite{haldane,BaakeBaakeSalamat} together with an additional term, and formulate the necessary conditions in a precise way.
Then, in Section~\ref{sec:duality}, we briefly recall the concept of duality for Markov processes, which we use to formalise the link between the evolution and the genealogical structure. There, we also state and prove our first main result; given a dual process for the evolution without recombination, we construct a labelled partitioning process that is dual to the evolution \emph{with} recombination. For this, the duality for the evolution without recombination needs to satisfy certain assumptions that reflect the assumptions on the evolution itself. In Section~\ref{sec:singlecrossover}, we explore the additional structure present in the case of single-crossover, providing an alternative description of the gLPP in terms of an independent collection of Markov processes with an additional starting / initiation mechanism.  In addition, we give a recursive solution in terms of iterated integrals. To not lose our grounding in the real world, we close in Section~\ref{sec:application} by applying our abstract results to the example of the \emph{selection-mutation recombination equation}.

\section{The $\psi$-recombination equation}\label{sec:psirecoeq}
We consider a large population of constant size, consisting of haploid individuals\footnote{In contrast to \emph{diploid} individuals, who carry two sets of chromosomes, \emph{haploid} individuals only carry one set. Although recombination takes place in the context of sexual reproduction and thus requires diploid organisms, this simplification can be justified by the assumption of Hardy-Weinberg equilibrium.}. Their genomes are thought of as linear arrangements of $n$ genetic sites, the set of which we denote by $S \defeq [n] \defeq [1:n] \defeq \{1,\ldots,n\}$. The allele at each site $i$ is represented by a letter $x_i^{} \in X_i$, where the alphabets $X_i$ may be distinct. For now, we will only assume that the alphabets are locally compact Hausdorff spaces; this generality will be useful in view of possible applications in the context of quantitative genetics, and does not require any additional effort; cf.~\cite[Ch.~IV]{BurgerBuch}.
Formally, the set of all genetic types is given by 
\begin{equation*}
X \defeq X_1 \times \ldots \times X_n.
\end{equation*}
Assuming that the population is sufficiently large so that we may neglect stochastic effects and normalising the constant population size, we model the evolution of the genetic type composition of this population by a continuous-time family $(\omega_t^{})_{t \geqslant 0}^{}$ of (Borel) probability measures. We will denote the set of all probability measures by $\cP(X)$. Later, we will need some language to describe the distribution of types that are only specified at a subset of sites. To this end, we define, for any 
 $A \subseteq S$, the set of \emph{marginal types} with respect to $A$ as
\begin{equation*}
X_A^{} \defeq \bigtimes_{i \in A} X_i.
\end{equation*}
Accordingly, we write $\cP(X_A)$ for the corresponding set of marginal \emph{distributions}. If $A = \varnothing$, then $X_\varnothing = \{\epsilon\}$, where $\epsilon$ should be thought of as the empty sequence. 
Finally, for any $A, B \subseteq S$ and any $\nu \in \cP(X_A)$, we denote by $\nu^{}_B$ the \emph{marginal distribution of $\nu$ with respect to $A \cap B$}, which is a (Borel) probability measure on $A \cap B$, defined via
\begin{equation*}
\nu_B^{}(\cdot) \defeq \nu \big (\cdot \times X_{S \setminus ( A \cap B )} \big ).
\end{equation*}
It is not difficult to see that this definition is consistent in the sense that for all subsets $A$, $B$ and $C$ of $S$ and all $\nu \in \cP(X_C)$, we have
\begin{equation*}
\big (\nu_A^{} \big )_B = \big (\nu_B^{} \big )_A = \nu_{A \cap B}^{} = \nu_{A \cap B \cap C}^{},
\end{equation*}
including the case that $A \cap B \cap C = \varnothing$.

We will consider a more general version of the recombination process compared to \cite{SelrekDocumenta} and \cite{SelrekTPB}, where we restricted our attention to single-crossover only. As in \cite{haldane} and \cite{BaakeBaakeSalamat}, we allow for multiple crossovers, and even an arbitrary number of parents. To state this in a concise way, we will need the notion of a \emph{partition} of the set of sequence sites.

Recall that a partition of a (nonempty) set $M$ is a collection of nonempty, disjoint subsets of $M$, called \emph{blocks}, whose union is $M$; we denote the set of all partitions of $M$ by $\bP(M)$.  To any $\cA \in \bP(S)$, we associate a \emph{recombination rate} $\varrho_\cA^{} \geqslant 0$.
Intuitively, this means that, at rate $\varrho_\cA^{}$, a group of $|\cA|$ parent individual, corresponding to the blocks of $\cA$ and chosen independently, uniformly and without replacement, comes together to produce a new offspring individual. We call such an offspring \emph{$\cA$-recombined}, and its type $y$ is given by  $y_i = x_{A(i),i}^{}$ where $A(i) \in \cA$ is the unique block of $\cA$ that contains $i$ and $x_{A(i),i}$ is the $i$-th component of the type of the parent associated with block $A(i)$. Due to our  assumption of random mating, the type of an $\cA$-recombined offspring individual (born, say, at time $t$) is distributed according to the product measure
\begin{equation} \label{recombinator}
\cR_\cA (\omega_t^{}) \defeq \bigotimes_{A \in \cA} \omega_{A,t}^{},
\end{equation}
where we write $\omega_{A,t}^{}$ instead of $(\omega_t^{})_A^{}$.
The operator $\cR_\cA : \cP(X) \to \cP(X)$ thus defined is called a \emph{recombinator}; see \cite{bb,haldane,recoreview}. 

Thus, the evolution of the type composition under the influence of recombination alone is captured by the \emph{general recombination equation}
\begin{equation} \label{purereco}
\dot{\omega}_t^{} = \sum_{\cA \in \bP(S)} \varrho_\cA^{} \big ( \cR_\cA^{}(\omega_t^{}) - \omega_t^{}    \big ),
\end{equation}
where the term $-\omega_t^{}$ accounts for the replacement of randomly chosen individuals by $\cA$-recombined offspring, thereby keeping the population size constant. This equation was solved in~\cite{BaakeBaakeSalamat} via a lattice-theoretic approach, and in~\cite{haldane} by exploiting its connection to a stochastic partitioning process; see also Section~\ref{sec:duality}.

\begin{remark}
In what follows, the following retrospective (from the offspring to the parents) interpretation of Eq.~\eqref{purereco} will be useful; namely that each individual is independently, at rate $\varrho_\cA^{}$, replaced by an $\cA$-recombined individual that chooses its parents independently. For a recent account of ancestral methods in the context of recombination, see~\cite{recoreview}. \hfill $\diamondsuit$
\end{remark}

We assume that the population is subject to additional evolutionary forces such as selection and/or mutation, which we assume to act independently of recombination and which are described by a Lipschitz continuous (with respect to the total variation norm) vector field $\psi$ on
$\cP(X)$. So, the complete time evolution of the type composition of our population is captured by the
\emph{$\psi$-recombination equation},
\begin{equation}\label{psirecoeq}
\dot{\omega}_t^{} = \psi(\omega_t^{}) +  \sum_{\cA \in \bP(S)} \varrho_\cA^{} \big ( \cR_\cA^{}(\omega_t^{}) - \omega_t^{}    \big ).
\end{equation}
For later use, we denote the flow associated to Eq.~\eqref{psirecoeq} by $\Psi = (\Psi_t^{})_{t \geqslant 0}^{}$. It will also be useful to introduce the flow generated by $\psi$ alone, denoted by $\Psi^{(0)} = (\Psi^{(0)}_t)_{t \geqslant 0}^{}$.
The Lipschitz continuity (with respect to the norm of total variation) of the recombinators~\cite[Prop.~1]{haldane} and of $\psi$ guarantee that Eq.~\eqref{psirecoeq} has a unique (global) solution for any initial condition.  
Moreover, we want $\psi$ to play nice with recombination, as detailed in Assumption~\ref{ass:psiassumption} below. Most importantly, we want to ensure that the way individuals are affected by evolution (apart from recombination) only depends on the allele at a fixed site $i_\bullet^{} \in S$. We call $i_\bullet^{}$ the \emph{active} site and set $S^\circ \defeq S \setminus \{i_\bullet^{}\}$.
\begin{assumption}\label{ass:psiassumption}
\begin{enumerate}[label=\textnormal{(\alph*)}]
\item \label{item:multiplicativity}
For all $\nu \in \cP(X)$ of the form $\nu = \nu_C^{} \otimes \nu_D^{}$ where $\{C,D\} \in \bP(S)$ and $i_\bullet^{} \in C$, 
\begin{equation*}
\psi (\nu) = \psi \big ( \nu_C^{} \otimes \nu_D^{}  \big ) = \psi (\nu)_C^{} \otimes \nu^{}_D.
\end{equation*}
\item \label{item:linearity}
$\psi$ acts linearly on subsets of elements of $\cP(X)$ that share the same marginal distribution at site $i_\bullet^{}$. That is, if $\mu$ and $\mu'$ are type distributions with $\mu_{\{i_\bullet^{}\}}^{} = \mu'_{\{i_\bullet^{} \}}$ and $\alpha \in [0,1]$, then
\begin{equation*}
\psi \big ( \alpha \mu + (1 - \alpha) \mu' \big ) = \alpha \psi (\mu) + (1 - \alpha) \psi (\mu').
\end{equation*} 
\end{enumerate}
\end{assumption}

In words, Assumption~\ref{ass:psiassumption}\ref{item:multiplicativity} states that if the set of sequence sites can be decomposed into two independent parts,  then this independence is preserved under the action of $\psi$. Moreover, only the part that contains the active site is affected by 
$\psi$. Condition~\ref{ass:psiassumption}\ref{item:linearity} means that any nonlinearity contained in $\psi$ depends only on the marginal distribution at $i_\bullet^{}$. 

Due to the particular role that is played in Assumption~\ref{ass:psiassumption}\ref{item:multiplicativity} by the part of a partition that contains $i_\bullet^{}$, we give it a special name.

\begin{definition} \label{def:head}
Let $\cA \in \bP(S)$. The \emph{head} of $\cA$ is defined as the (unique) block $A_\bullet \in \cA$ with $i_\bullet^{} \in A_\bullet$. The complement of $\{A_\bullet\}$, the collection $\cA \setminus \{A_\bullet\}$, is denoted by $\cA^\circ$ and referred to as the \emph{tail} of $\cA$.
\end{definition}

Let now $\cA = \{A_1,\ldots,A_n\} \in \bP(S)$ with head $A_\bullet$. Then, an iterative application
 of Assumption~\ref{ass:psiassumption}\ref{item:multiplicativity} and the definition of the recombinator in Eq.~\eqref{recombinator} shows that
\begin{equation} \label{onlyonthehead}
\psi \big ( \cR_{\cA} (\nu) \big ) = \psi(\nu)_{A_\bullet}^{} \otimes \bigotimes_{A \in \cA \setminus \{A_\bullet\}} \nu_A^{}.
\end{equation}
Here and in the following, we assume products such as on the right-hand side of Eq.~\eqref{onlyonthehead} to be implicitly site-ordered.

\section{Duality}\label{sec:duality}
For the convenience of the reader, we briefly recall the notion of duality for Markov processes. Generally speaking, two Markov processes $Y = (Y_t)_{t \geqslant 0}$ and $Z = (Z_t)_{t \geqslant 0}$ are said to be \emph{dual} with respect to a \emph{duality function} $h$ if 
\begin{equation}\label{abstractduality}
\EE \big [h(y,Z_t) \mid Z_0 = z \big ] = \EE \big [h(Y_t,z) \mid Y_0 = y \big ]
\end{equation}
for all choices of initial conditions $z$ and $y$, implicitly assuming integrability. In mathematical population genetics, this concept is used extensively to formalise the link between evolutionary dynamics and ancestral processes; see \cite[Ch.~3.4.4]{liggett} and \cite{kurtjansen} for a thorough exposition, and \cite{Moehle99} for early applications in the context of population genetics. While the classical theory is mainly concerned with scalar-valued duality functions, the generalisation to vector-valued duality functions, as needed here\footnote{In fact, our focus will be on duality  functions that take values in $\cP(X)$.
Keep in mind that the space of \emph{signed} Borel measures on $X$, equipped with the norm of total variation, is a Banach space. Thus, even for uncountable $X$, the expectation can be understood in the sense of Bochner integration.}, is straightforward. We will throughout abbreviate dualities of the form~\eqref{abstractduality} as triples $(Y,Z,h)$.

To fit our discussion into this framework, we interpret the solution of a differential equation as a Markov process with deterministic transitions. In this spirit, the (general) solution of Eq.~\eqref{psirecoeq} can be viewed as a Markov process with transition semigroup $P = (P_t)_{t \geqslant 0}$ given by
\begin{equation*}
P_t^{}(\nu,\cdot) \defeq \delta_{\Psi_t(\nu)}^{}.
\end{equation*}
Similarly, in the absence of recombination, we interpret the solution $\omega^{(0)} = (\omega_t^{(0)})_{t \geqslant 0}^{}$ of 
\begin{equation} \label{onlypsi}
\dot \omega_t^{(0)} = \psi \big ( \omega_t^{(0)} \big ),
\end{equation}
as a Markov process with transition semigroup $P^{(0)}$, given by
\begin{equation*}
P_t^{(0)} (\nu,\cdot) \defeq \delta_{\Psi^{(0)}_t(\nu)}^{}.
\end{equation*}
It is obvious that the transition semigroups $P$ and $P^{(0)}$ are Feller. Moreover, their infinitesimal generators $\cG$ and $\cG^{(0)}$ are defined on $\cC^1 \big (  \cP(X), \RR \big  )$, the set of continuously differentiable, real-valued functions on $\cP(X)$ via
\begin{equation*}
\cG \varphi (\nu) \defeq \lim_{h \to 0} \frac{1}{h} \big [ \varphi \big (\Psi_h (\nu) \big ) - \varphi(\nu) \big ] 
= \lim_{h \to 0} \frac{1}{h} \Big [\varphi \Big (\nu + h \psi(\nu) + h \sum_{\cA \in \bP(S)} \varrho_\cA^{} \big ( \cR_\cA(\nu)    - \nu  \big ) \Big )  - \varphi(\nu) \Big ]
\end{equation*}
and
\begin{equation*}
\cG^{(0)} \varphi (\nu) \defeq \lim_{h \to 0} \frac{1}{h} \big [ \varphi \big ( \Psi^{(0)}_h (\nu) \big ) - \varphi(\nu) \big ]
= \lim_{h \to 0} \frac{1}{h} \big [ \varphi \big ( \nu + h \psi (\nu)  \big )    - \varphi (\nu)       \big].
\end{equation*}
Note that $\cG \varphi (\nu)$ ($\cG^{(0)} \varphi (\nu)$) is simply the directional derivative of $\varphi$ at $\nu$ with respect to
\mbox{$\psi(\nu) + \sum_{\cA \in \bP(S)} \varrho_\cA^{} (\cR_\cA(\nu) - \nu )$} (with respect to $\psi(\nu)$).

Later, we will see how, under appropriate assumptions, we can construct a duality $(\omega,\Sigma,H)$ for Eq.~\eqref{psirecoeq} from a duality $(\omega^{(0)},Y,h)$ for Eq.~\eqref{onlypsi}. But first, we consider a few examples for the latter.
\begin{example} [deterministic flows] \label{ex:semigroup} 
Let $\psi$ be arbitrary, satisfying Assumption~\ref{ass:psiassumption}, and let $Y = (Y_t)_{t \geqslant 0}$ be $\RR_{\geqslant 0}$-valued and deterministic with  $Y_t \defeq Y_0 + t$. The duality function $h$ is given by the flow associated with $\psi$ alone, that is, $h(\nu,y) \defeq \Psi^{(0)}_y (\nu)$ for all $\nu \in \cP(X)$ and $y \in \RR_{\geqslant 0}$.
\end{example}

\begin{example}[selection] \label{ex:pureselection}
Assume that $\psi$ describes frequency-independent selection as in \cite{SelrekDocumenta,SelrekTPB}, that is, $X = \{0,1\}^n$ and
\begin{equation}\label{psiseldefprime}
\psi (\nu)(x) \defeq s \big ( (1 - x_{i_\bullet^{}}^{}) \nu(x) - f(\nu) \nu(x) \big );
\end{equation}
see also Section~\ref{sec:application}. Here, $f(\nu) \defeq \nu_{i_\bullet^{}}^{} (0)$ is the frequency of fit individuals in a population with type distribution $\nu$; an individual of type $x$ is called \emph{fit} if $x_{i_\bullet^{}}^{} = 0$, and  \emph{unfit} otherwise. Tracking simply the frequency $y(t)$ of the unfit type by integrating Eq.~\eqref{psiseldefprime}  with respect to the indicator $\one_{x_{i_\bullet = 1}}^{}$ yields 
\begin{equation*}
\dot y(t) = -s y(t) \big (1  - y(t)  \big ),
\end{equation*}
which is the classical selection equation for $n=1$.
 
An intuitive interpretation of Eq.~\eqref{psiseldefprime} is that each individual in the population is, at rate $s$, replaced by the offspring of a fit individual. It was shown in \cite{SelrekDocumenta} that the solution of Eq.~\eqref{onlypsi} with $\psi$ as in~\eqref{psiseldefprime} is dual to the line counting process 
$K = (K_t)_{t \geqslant 0}$ of a Yule process with (binary) branching rate $s$, which can be viewed as the line counting process of the \emph{ancestral selection graph} (ASG)~\cite{KroneNeuhauser97} in the deterministic limit~\cite{BaakeCorderoHummel18}. This process counts the number of potential ancestors of an individual sampled from the current population that were alive at backward time $t$. This individual is unfit, if and only if all its potential ancestors were unfit. Thus, if it has $k$ potential ancestors whose types are independently distributed according to $\nu$, its type is distributed according to 
\begin{equation*}
h(\nu,k) \defeq \big (1 - f(\nu) \big )^k d(\nu) + \big ( 1 - (1 - f(\nu))^k \big ) b(\nu),
\end{equation*}
where $b(\nu)$ ($d(\nu)$) is the type distribution within the subpopulation of fit (unfit) individuals, and $f(\nu)$ is the frequency of fit individuals in a population with type distribution $\nu$.

 For the generalisation to frequency-dependent selection and mutation at a single locus, we refer the reader to \cite{BaakeCorderoHummel21}.
\end{example}

\begin{example}[mutation] \label{ex:mutation}
Assume that $X = \{0,1\}^n$ and 
\begin{equation*}
\psi(\nu) \defeq \sum_{i \in [n]} u_i^{} \sum_{x \in X} \nu(x) \big ( m_{i,1} \delta_{(x_1^{},\ldots,x_{i-1}^{},1,x_{i+1}^{},\ldots,x_n^{})}^{}   +  m_{i,0} \delta_{(x_1^{},\ldots,x_{i-1}^{},0,x_{i+1}^{},\ldots,x_n^{})}^{} - \delta_x^{}  \big ).
\end{equation*}
In words, mutations occur independently at each site $i$ at rate $u_i^{}$. Upon a mutation at site $i$, the allele at this site mutates to $1$ with probability $m_{i,1}^{}$ and to $0$ with probability $m_{i,0}^{}$, regardless of the orginal type.
If a given site did not experience mutation on the ancestral lineage of a sampled individual, the letter at this site is copied from its ancestor. Otherwise, it is determined by the last (or \emph{first}, when looking backward in time) mutation. We therefore need to keep track, at each site independently, of this first mutation, or of the fact that no mutation has occurred.
Thus, the dual process is, in this case, given by an independent collection of continuous-time Markov chains $Y_1,\ldots,Y_n$ with state space $\{0,1,\circ\}$. For each $i \in [n]$, $Y_{i,t} = \circ$ means that until (backward) time $t$, site $i$ has remained unaffected by mutation while $Y_{i,t} = 0$ ($=1$) means that a mutation has changed the allele to $0$ (to $1$). Accordingly, each $Y_i$ transitions from $\circ$ to $0$ and $1$ with rates $u_i^{} m_{i,0}^{}$ and $u_i^{} m_{i,1}^{}$, respectively, and the states $0$ and $1$ are absorbing. Given $y \in \{0,1,\circ\}^n$ and $\nu \in \cP(X)$, we define $h(\nu,y)$ via the following random experiment. First, draw a sample $x$ according to $\nu$. Then, for all $i \in [n]$ with $y_i^{} \neq \circ$, change $x_i^{0}$ to $y_i^{}$. More formally, 
\begin{equation*}
h(\nu,y)(x) \defeq h(\nu,y_1^{},\ldots,y_n^{})(x) \defeq \nu_B^{} (x_B^{}) \prod_{i \in S \setminus B} \delta_{x_i^{},y_i{}}^{},
\end{equation*}
where $B$ is the set of sites $i$ for which $y_i^{} = \circ$ and $\delta$ denotes the Kronecker delta. We will elaborate on this example in Section~\ref{sec:application}.
\end{example}

\begin{remark} \label{notmultiplicative}
Note that in Exs.~\ref{ex:semigroup} and~\ref{ex:pureselection}, $\psi$ satisfies Assumption~\ref{ass:psiassumption}. For Ex.~\ref{ex:semigroup}, this is immediate from Assumption~\ref{ass:psiassumption}, while for Example~\ref{ex:pureselection}, this was alluded to in~\cite[Remark 7.14 (i)]{SelrekDocumenta} and will be proven in Section~\ref{sec:application}. However, $\psi$ in Ex.~\ref{ex:mutation} does \emph{not} satisfy Assumption~\ref{ass:psiassumption}; we will address this problem in Section~\ref{sec:application}.
\end{remark}

Observe that in all three examples, there exists a distinguished state $\emptyset$ (not to be confused with $\varnothing$ which denotes the empty set) for the dual process such that $h(\cdot,\emptyset)$ is the identity on $\cP(X)$;
namely $\emptyset = 0$ in Ex.~\ref{ex:semigroup}, $\emptyset = 1$ in Ex.~\ref{ex:pureselection} and $\emptyset =  (\circ,\circ,\ldots,\circ)$ in Ex.~\ref{ex:mutation}.  In each example, the duality relation~\eqref{abstractduality} thus contains the stochastic representation
\begin{equation}\label{stochasticrepresentation}
\omega_t^{(0)} = h(\omega_t^{(0)},\emptyset) = \EE [h( \omega_0^{(0)},Y_t) \mid Y_0 = \emptyset]
\end{equation}
of the solution of Eq.~\eqref{onlypsi} as a special case. Generally, we interpret $h(\omega_0^{},y)$ as the distribution of the type of an individual whose genealogy is described by $y$, where $\omega_0^{}$ describes the type distribution within the ancestral generation. Therefore, the state $\emptyset$ can be viewed as an `empty genealogy'  of an individual whose ancestry has remained unaffected by the evolutionary forces modelled by $\psi$.
Recall that in Example~\ref{ex:semigroup}, Assumption~\ref{ass:psiassumption} immediately implies Assumption~\ref{ass:psiassumption} for $h(\cdot,t)$ in place of $\psi$ for all $t \geqslant 0$. As we will see, this is the key property that allows to lift the duality 
$(\omega^{(0)},Y,h)$ to the case with recombination. In addition, we need a few technical assumptions.
\begin{assumption}\label{ass:hassumption}
The duality $(\omega^{(0)},Y,h)$ satisfies the following conditions.
\begin{enumerate}[label=\textnormal{(\alph*)}]
\item
$Y$ is a Feller process on a Polish state space $E$ with semigroup $Q = (Q_t^{})_{t \geqslant 0}^{}$ and infinitesimal generator $\cQ$.
\item \label{item:hlinearity}
For all $y \in E$, $h(\cdot,y)$ satisfies Assumption~\textnormal{\ref{ass:psiassumption}}.
\item \label{item:emptygen}
There is $\emptyset \in E$ such that $h(\nu,\emptyset) = \nu$ for all $\nu \in \cP(X)$.
\item \label{item:diffbar}
For all $y \in E$, $h(\cdot,y) \in \cC^1 \big (\cP(X) \big )$ and for all $\nu \in \cP(X)$, $h(\nu,\cdot) \in \cD(\cQ)$, where $\cD(\cQ)$ is the domain of the generator $\cQ$ of $Y$. 
\end{enumerate}
\end{assumption}
We assume differentiability in \ref{item:diffbar} to guarantee that $h(\cdot,y)$ is in the domain of the generator $\cG^{(0)}$ of $\omega^{(0)}$.

\subsection{Including recombination} \label{subsec:addrecombination}
In the following, we assume that $(\omega^{(0)},Y,h)$ satisfies Assumption~\ref{ass:hassumption}. 
As announced earlier, we want to lift $(\omega^{(0)},Y,h)$ to a duality $(\omega,\Sigma,H)$ for Eq.~\eqref{psirecoeq}; in essence, $\Sigma$ will be a partitioning process whose blocks are labelled by independent copies of $Y$.

It is well known~\cite{haldane,recoreview} that the pure recombination equation is dual to a partitioning process 
$\Sigma^{(0)} = (\Sigma^{(0)}_t)_{t \geqslant 0}$, which takes values in $\bP(S)$ and keeps track of the fragmentation of an individual's genome across its ancestors. More explicitly, the blocks of $\Sigma^{(0)}_t$ correspond to different parts of the genome that are inherited from independent ancestors alive at backward time $t$ (forward time $0$). Therefore, given a realisation of the partitioning process and the initial type distribution $\nu$, the type distribution at time $t$ is given by
\begin{equation}\label{purerecorep}
\bigotimes_{\sigma \in \Sigma_t^{(0)}} \nu_\sigma^{}. 
\end{equation}
The evolution of $\Sigma^{(0)}$ can be described as follows.
 Independently of each other and with rate $\varrho_\cB^{}$ for any $\cB \in \bP(S)$, every block $\sigma$ is subdivided into the collection
\begin{equation*}
\{ \sigma \cap B : B \in \cB, \, \sigma \cap B \neq \varnothing \} \defeq \cB|_{\sigma}
\end{equation*} of blocks, which we call the partition that $\cB$ induces on $\sigma$; we also say that the block $\sigma$ is hit by a  \emph{$\cB$-splitting}. If the intersection is empty for all but one $B \in \cB$, the transition is silent.  Intuitively, if $\sigma$ is at backward time $t$ hit by a $\cB$-splitting, this means that the ancestor of $\sigma$, that was alive at time $t-0$, is itself a $\cB$-recombined offspring of parents that were alive at backward time $t$. 

To combine this with the duality $(\omega^{(0)},Y,h)$ for the case without recombination, we define a \emph{labelled} partitioning process $\Sigma$ by associating to each block an independent copy of $Y$ that describes the evolution of the ancestry of the corresponding ancestor \emph{under the action of $\psi$ alone}. Given the initial type distribution $\nu$ and a realisation of the labelled partitioning process $\Sigma$, we see that the type distribution at time $t$ is, in analogy with Eq.~\eqref{purerecorep} and recalling our earlier interpretation of the duality function $h$, given by
\begin{equation*}
\bigotimes_{\sigma \in \Sigma_t^{}} h(\nu,y_\sigma^{})_\sigma^{},
\end{equation*}
where $y_\sigma^{} \in E$ is the label associated with the block $\sigma$.

To fully describe the evolution of $\Sigma$, we still need to understand the effect that the fragmentation of blocks has on the labels.
Assume for instance that the block $\sigma$ is, say at time $t$, hit by a $\cB$-splitting. This means that the ancestor of $\sigma$ is part of a subpopulation in which the blocks of $\cB$ evolved independently of each other. Because we are assuming that the evolutionary forces that are modelled by $\psi$ (forward in time) and by $Y$ (backward in time) have only acted on the head $B^{}_\bullet$ of $\cB$, the marginal distribution with respect to $\sigma$ is given by 
\begin{equation} \label{marg354e}
h(\nu,Y_{\sigma,t-0})_{\sigma \cap B_\bullet^{}}^{} \otimes \bigotimes_{B \in \cB^\circ} \nu_{\sigma \cap B}^{} =
h(\nu,Y_{\sigma,t-0})_{\sigma \cap B_\bullet^{}}^{} \otimes \bigotimes_{B \in \cB^\circ} h(\nu,\emptyset)_{\sigma \cap B}^{},
\end{equation}
where the equality is due to Assumption~\ref{ass:hassumption}~\ref{item:emptygen}; this is the backward-time analogue of Eq.~\eqref{onlyonthehead}.
Comparing with Eq.~\eqref{marg354e}, we see that as $\sigma$ is subdivided into the blocks $\sigma \cap B$, the label $Y_{\sigma,t-0}$ associated to $\sigma$ just before the splitting is inherited by $\sigma \cap B_\bullet^{}$ (if nonempty) while new independent copies of $Y$, starting from $\emptyset$, are associated to the intersections of $\sigma$ and the blocks in $\cB^\circ$. In particular, note that this means that the transition is silent if $\sigma \subseteq B^{}_\bullet$, and that the transition amounts to `resetting' the label of $\sigma$ to $\emptyset$ if $\sigma \subseteq B$ for some $B \in \cB^\circ$. To summarise, defining a labelled partition as a pair $(\cA,v)$ where $\cA \in \bP(S)$ and $v = (v_A^{})_{A \in \cA}^{} \in E^\cA$ is the vector of labels, the labelled partitioning process can be informally described as follows (compare~\cite[Def.~7.3]{SelrekDocumenta}).
\begin{enumerate}[label=(\arabic*)]
\item \label{trans:splitting}
$(\cA,v) \longrightarrow \big ((\cA \setminus \{A\}) \cup \{A \cap B : B \in \cB, \, A \cap B \neq \varnothing \},w \big )$ at rate $\varrho_\cB^{}$ for all $\cB \in \bP(S)$, all $A \in \cA$ such that $A \not \subseteq B$ for any $B \in \cB$ and $w$ such that $w_{A \cap B^{}_\bullet}^{} = v_A^{}$ (if nonempty) and $w_{A \cap B}^{} = \emptyset$ for all $B \in \cB^\circ$. 
\item \label{trans:resetting}
$(\cA,v) \longrightarrow (\cA,w)$ at rate 
\begin{equation*}
r_A^{} \defeq \sum_{\cB \in \bP(S) : A \subseteq B \textnormal{ for some } B \in \cB^\circ} \varrho_\cB^{}
\end{equation*}
for each $A \in \cA$ and $w$ such that $w_A^{} = \emptyset$ and $w_C^{} = v_C^{}$ for all $C \in \cA \setminus \{A\}$.
\item \label{trans:otherwise}
In between transition of type~\ref{trans:splitting} or~\ref{trans:resetting}, the labels evolve independently according to the transition semigroup of $Y$.
\end{enumerate}
\begin{remark}\label{rmk:onlyacopy}
Note that in \ref{trans:splitting}, the head of the partition on the right-hand side is given by $A_\bullet^{} \cap B_\bullet^{}$, so that the label of the head is not affected by splittings. Similiarly in \ref{trans:resetting}, we always have $r_{A^{}_\bullet}^{} = 0$. Therefore, by \ref{trans:otherwise}, the label of the head is simply a copy of $Y$.
\end{remark}

In order to state the generator of the labelled partitioning process in a concise manner (which will be needed for the formal proof of the duality), it is useful to introduce some additional notation. Given a labelled partition $(\cA,v)$ and a(n) (unlabelled) partition $\cB \in \bP(S)$, we define $(\cA,v) \boxwedge_{A} \cB$ to be the labelled partition of $S$ that is constructed by replacing the block $A$ of $\cA$ by the collection $A \cap B$ for all $B \in \cB$ such that $A \cap B$ is nonempty. For each block of the form  $A \cap B$, the associated label $v_{A \cap B}^{}$ is given by $v_A^{}$ if $i_\bullet^{} \in B$, and $\emptyset$ otherwise. Now, transitions~\ref{trans:splitting}--\ref{trans:resetting} can be summarised as follows: if the labelled partitioning process is currently in state $(\cA,v)$, it transitions to $(\cA,v) \boxwedge_{A} \cB$ at rate $\varrho_\cB^{}$, independently for each $A \in \cA$ and $\cB \in \bP(S)$.

\begin{definition}\label{def:partitioningprocess}
The \emph{general labelled partitioning process} (gLPP) is a continuous-time Markov chain $\Sigma = (\Sigma^{(0)}_t,V_t)_{t \geqslant 0}$ with state space
\begin{equation*}
F = \bigcup_{\cA \in \bP(S)} \big (\{\cA\} \times E^\cA \big)
\end{equation*}
and generator $\cL$, which is defined on 
\begin{equation*}
\widetilde \cD (\cL) \defeq \Big \{\varphi :  \, \varphi \big ( (\cA,v_i^{},\ldots,v_{i-1}^{},\cdot,v_{i+1}^{},\ldots,v_{|\cA|}^{} )  \big ) \in \cD(\cQ) \textnormal{ for all } 1 \leqslant i \leqslant |\cA| \textnormal{ and } v \Big \}
\end{equation*}
via 
\begin{equation*}
\cL \varphi \big ( (\cA,v)   \big ) = \sum_{A \in \cA} \sum_{\cB \in \bP(S)} \varrho_\cB^{} \Big ( \varphi \big ( ( \cA,v) \boxwedge_{A} \cB \big )  
- \varphi \big ( (\cA, v) \big ) \Big ) + \sum_{j = 1}^{|\cA|} \cQ_i \varphi \big ( (\cA,\cdot)    \big )(v),
\end{equation*}
where $v_i^{}$ is the label of the $i$-th block in the usual ordering 
and $\cQ_i$ is the generator of the dual Feller process $Y$, acting on the $i$-th component of the function $\varphi$ restricted to 
$\{\cA\} \times E^{\cA} \cong E^{|\cA|}$. More precisely,
\begin{equation*}
\cQ_i \varphi \big ( (\cA,\cdot)  \big ) (v) \defeq \cQ \varphi \big (  (\cA,v_{A_1}^{},\ldots,v_{A_{i-1}}^{},\cdot,v_{A_{i+1}}^{},\ldots,v_{A_{|\cA|}}^{} ) \big )(v_{A_i}^{}).
\end{equation*}
\end{definition}

It can be shown that F can be endowed with a topology so that it becomes a Polish space and that $\cL$ generates a (unique) Feller semigroup; thus the gLPP is well defined. The proof is straightforward but cumbersome and therefore relegated to the appendix.
\begin{lemma}\label{lem:existence}
The \textnormal{gLPP} in Definition~\textnormal{\ref{def:partitioningprocess}} is well defined. More precisely, $F$ is a Polish space and there exists a unique Feller process 
$\Sigma = (\Sigma_t^{(0)},V_t)_{t \geqslant 0}$ with state space $F$ and generator $\cL$.
\end{lemma}
\begin{proof}
See Appendix~\ref{app:proof}.
\end{proof}

We now formulate the main result of this section. 

\begin{theorem}\label{thm:duality}
Assume that the duality $(\omega^{(0)},Y,h)$ holds and that Assumption~\textnormal{\ref{ass:hassumption}} is satisfied. Then, the duality $(\omega,\Sigma,H)$ holds, where
\begin{equation*}
H(\cA,v;\nu) \defeq \bigotimes_{A \in \cA} h(\nu,v_A^{})^{}_A,
\end{equation*}
and $\Sigma$ is the \textnormal{gLPP} according to Definition~\textnormal{\ref{def:partitioningprocess}}.
In particular, this entails the stochastic representation 
\begin{equation*}
\omega_t^{} = \EE \big [ H(\Sigma_t^{(0)},V_t;\omega_0^{} ) \mid \Sigma_0^{(0)} = \{S\}, V_{S,0} = \emptyset \big ].
\end{equation*}
\end{theorem} 

\begin{proof}
We will use~\cite[Proposition 1.2]{kurtjansen}\footnote{Strictly speaking, their result was proved only for scalar valued duality functions. Alternatively, one could invoke~\cite[Prop.~2.10]{EundK}; note that $H$ is continuous.}. Since the semigroups of both processes are Feller, it suffices to show that
\begin{equation*}
\cG H(\cA,v; \cdot) (\nu) = \cL H (\cdot, \nu) \big ( (\cA,v)  \big ) 
\end{equation*}
for all $(\cA,v) \in F$ and $\nu \in \cP(X)$. By the product rule, it suffices to study the action of $\cG$ on the individual factors that make up $H$. Thus, we start by computing $\cG h (y,\cdot) (\nu)$ for arbitrary $y \in E$.
Observing that $\Psi_\delta^{(0)} (\nu) = \nu + \delta \psi(\nu) + O(\delta^2)$, we see that
\begin{equation*}
\begin{split}
\Psi_\delta(\nu) &=  \nu + \delta \sum_{\cB \in \bP(S)} \varrho_\cB^{} \big (\cR_\cB (\nu) - \nu \big) + \delta \psi (\nu)  + O(\delta^2)\\ 
&= \Big ( 1 - \delta \sum_{\cB \in \bP(S)} \varrho_\cB^{} \Big ) \Psi^{(0)}_\delta(\nu) + \delta \sum_{\cB \in \bP(S)} \varrho_\cB^{} \Psi^{(0)}_\delta(\nu)_{B_\bullet^{}}^{} \otimes \bigotimes_{B \in \cB^\circ} \nu_B^{} + O(\delta^2).
\end{split}
\end{equation*}
Together with the assumed differentiability of $h$ in the second argument, this implies 
\begin{equation*}
h \big (y,\Psi_\delta^{} (\nu) \big ) = h \Big (v, \Big (1 - \delta \sum_{\cB \in \bP(S)} \varrho_\cB^{} \Big ) \Psi^{}_\delta (\nu) +  \delta \sum_{\cB \in \bP(S)} \varrho_\cB^{} \Psi^{}_\delta(\nu)_{B_\bullet^{}} \otimes
\bigotimes_{B \in \cB^\circ} \nu_B^{} \Big ) + o (\delta).
\end{equation*}
Invoking first Assumption~\ref{ass:hassumption}\ref{item:hlinearity} and then Assumption~\ref{ass:hassumption}\ref{item:multiplicativity}, this becomes
\begin{equation*}
h \big ( y, \Psi_\delta^{} (\nu) \big ) = \Big ( 1 - \delta \sum_{\cB} \varrho_\cB^{} \Big ) h \big (v,\Psi_\delta^{}(\nu) \big) 
+ \delta \sum_{\cB} \varrho_\cB^{} \, h \big (v,\Psi_\delta^{}(\nu) \big )_{B_\bullet^{}} \otimes \bigotimes_{B \in \cB^\circ} h(\emptyset, \nu)_B + O(\delta^2).
\end{equation*}
Subtracting $h(y,\nu)$ on both sides, dividing by $\delta$ and subsequently letting $\delta \to 0$, we see that
\begin{equation*}
\begin{split}
\cG h(y,\cdot) (\nu) &= \cG^{(0)} h (y,\cdot) (\nu) + \sum_{\cB} \varrho_\cB^{} \Big ( h(y,\nu)_{B_\bullet^{}} \otimes \bigotimes_{B \in \cB^\circ} h(\emptyset, \nu)_B - h(y,\nu) \Big ) \\
&= \cQ h(\cdot,\nu)(y) + \sum_{\cB} \varrho_\cB^{}  \Big ( h(y,\nu)_{B_\bullet^{}} \otimes \bigotimes_{B \in \cB^\circ} h(\emptyset, \nu)_B - h(y,\nu) \Big ),
\end{split}
\end{equation*}
where the second equality follows from an application of the duality $(\omega^{(0)}, Y, h)$.

Keeping in mind that $\cG$ is a differential operator, we obtain by the product rule that 
\begin{equation*} 
\begin{split}
\cG &H(\cA,v;\cdot) (\nu) \\
&= \sum_{j = 1}^{|\cA|} \cQ_i H(\cA,\cdot ; \nu) (v_{A_j}^{})  \\
& \, \, \, \, \, + \sum_{A \in \cA} \sum_{\cB} \varrho_\cB^{} \Big (   H(\cA,v;\nu)_{S \setminus A} \otimes h(v_A^{},\nu)_{A \cap B_\bullet^{}}  \bigotimes_{B \in \cB^\circ} h(\emptyset, \nu)_{A \cap B} - H(\cA,v;\nu)   \Big ) \\ 
&=  \sum_{j = 1}^{|\cA|} \cQ_i H(\cA,\cdot ; \nu) (v_{A_j}^{}) \\
& \, \, \, \, \, + \sum_{A \in \cA} \sum_{\cB} \varrho_\cB^{}  \Big ( H\big ((\cA,v) \boxwedge_A \cB,\nu \big ) - H(\cA,v;\nu) \Big ) \\
&= \cL H(\cdot,\nu) \big ( (\cA,v)  \big ).
\end{split} 
\end{equation*}
\end{proof}

\section{The case of single-crossover}\label{sec:singlecrossover}
After having proved the general duality between the $\psi$-recombination equation and the gLPP, we explore the special case of single-crossover recombination. In this case, it turns out that the gLPP can be reduced to a collection of independent copies of $Y$, with
an additional starting/resetting mechanism. This simultaneously generalises both the representation of the weighted partitioning process from~\cite{SelrekDocumenta} in terms of a family of Yule processes with initiation and resetting (YPIR) and the duality to a collection of independent initiation processes.  As in~\cite{SelrekDocumenta}, a crucial ingredient is the following partial order on the set of sequence sites.

\begin{definition}\label{def:porder}
For any $i,j \in S$, we say that $i$ \emph{precedes} $j$ ($j$ \emph{succeeds} $i$) and write $i \preccurlyeq j$, if and only if either $i_\bullet^{} \leqslant i \leqslant j \leqslant n$ or
$1 \leqslant j \leqslant i \leqslant i_\bullet^{} \leqslant n$. In words, $i \preccurlyeq j$ means that $j$ is more distant from the active site than $i$. Note this entails that $i_\bullet^{} \preccurlyeq j$ for any $j \in S$.
\end{definition}

Using this partial order, any $i \in S^\circ$ defines a partition of $S$ into two blocks, namely the set $D_i$ of all sites that succeed $i$, and its complement, which we denote by $C_i^{}$. Clearly, $C_i$ is the head of the partition $\{C_i,D_i\}$ in the sense of Definition~\ref{def:head}, whence we call $C_i$ the $i$-head for short. The tail is given by the partition $\{D_i\}$ consisting of a single block $D_i$, which, by slight abuse of notation, we will refer to as the $i$-tail, in line with the nomenclature in~\cite{SelrekDocumenta}; see Fig.~\ref{fig:headsntails} for an illustration. More explicitly, we have
\begin{equation*}
(C_i,D_i) = \begin{cases}
\big ([1,\ldots,i-1],[i,\ldots,n] \big ) \textnormal{ if } i > i_\bullet^{}, \\
\big ([i+1,\ldots,n],[1,\ldots,i] \big )  \textnormal{ if } i < i_\bullet^{}.
\end{cases}
\end{equation*}
\begin{figure}[t]
\includegraphics[width=0.85\textwidth]{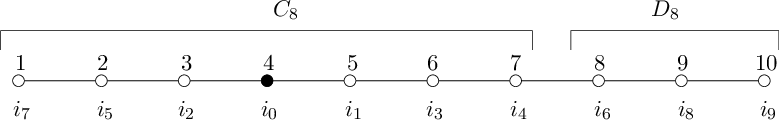}
\caption{\label{fig:headsntails}
A sequence of length $n=10$ where the active site is $i_\bullet^{} = 4$. The $8$-tail is given by $D_8 = \{8,9,10\}$, and the $8$-head by 
$C_8 = \{1,2,3,4,5,6,7\}$. The indices below the sequence indicate a non-decreasing reordering of the sites in the sense of Definition~\ref{def:porder}.
}
\end{figure}
With this, Eq.~\eqref{psirecoeq} can be rewritten as
\begin{equation*}
\dot \omega_t^{} = \psi(\omega_t^{}) + \sum_{i \in S^\circ} \varrho_i^{} \big ( \omega_{C_i,t}^{} \otimes \omega_{D_i,t}^{} - \omega_t^{} \big ),
\end{equation*}
where we abbreviated $\varrho_{\{C_i,D_i\}}^{}$ as $\varrho_i^{}$. 

It is clear that under the assumption of single-crossover, the gLPP takes values in the labelled \emph{interval} partitions of $S$. The aforementioned reformulation of the gLPP relies on a bijection between labelled interval partitions and labellings of the sites by elements of 
$\widetilde E \defeq E \cup \{\Delta\}$, given in Definition~\ref{def:bijection} and illustrated in Fig.~\ref{fig:illustration_of_encoding}; the new symbol $\Delta$ encodes the block structure.

\begin{definition}\label{def:bijection}
Given $(y_i^{})^{}_{i \in S}$ with $y_i^{} \in E \cup \{\Delta\}$ for all $i$, we construct a partition by splitting the set of sequence sites at those sites that are assigned elements of $E$. More precisely, we set
\begin{equation*}
\cA \defeq \bigwedge_{i \in S^\circ, \, y_i^{} \neq \Delta} \{C_i,D_i\},
\end{equation*}
where $\cB \wedge \cC \defeq \{B \cap C : B \in \cB,\, C \in \cC \} \setminus \{\varnothing\}$ denotes the coarsest common refinement of the partitions 
$\cB$ and $\cC$. 
By construction, each block $A$ of $\cA$ contains a unique site $i_A^{}$ that is labelled by an element of $E$, namely 
$i_A^{} = \min (A)$ in the sense of the partial order $\preccurlyeq$. Thus, $y_{i_A^{}}^{}$ becomes the label of block $A$. 

Conversely, starting from a labelled interval partition, we assign the label of the head of the partition to $i_\bullet^{}$. Then, for each remaining block, we assign the label of that block to the minimum (again in the sense of the partial order `$\preccurlyeq$') of that block. Finally, we assign a `$\Delta$' to each remaining site.  
\end{definition}

This correspondence allows us to translate the gLPP into a collection of processes $Y_{i} = (Y_{i,t})_{t \geqslant 0}$, $i \in S$, which take values in the augmented state space $\widetilde E$ for $i \in S^\circ$, and in $E$ for $i = i_\bullet^{}$. To understand their evolution, we start by recalling that as $Y_{i_\bullet^{}}$ represents the label of the head of the partition, we have $Y_{i_\bullet^{}} = Y$ in distribution; see Remark~\ref{rmk:onlyacopy}. For $i \in S^\circ$, the following transitions can occur, given that the current state is $y_i^{}$.
\begin{enumerate}
\item
If $y_i^{} \in E$, let $A_i$ be the block to which $i$ belongs. Then, the label of block $A_i$ is reset, i.e. $y_i^{}$ jumps to $\emptyset$, with rate $r_{A_i}^{}$. Because $i$ is the minimum of $A_i$ with respect to the partial order `$\preccurlyeq$', this rate depends only on $i$. Indeed,
\begin{equation*}
r_{A_i}^{}  = \sum_{j \in S^\circ : A_i \subseteq D_j} \varrho_j^{} = \sum_{j \in S^\circ : j \preccurlyeq i} \varrho_j^{} \eqdef r_i^{}.
\end{equation*}
So, when in $E$, the processes $Y_i^{}$, $i \in S^\circ$, jump to $\emptyset$, independently of each other and at rate $r_i^{}$, respectively. \newpage
\item
If $y_i^{} = \Delta$, then $i$ belongs to a block $A$ whose label is given by $y_j^{}$, where $j$ is the maximal site with respect to $\preccurlyeq$ for which
$j \preccurlyeq i$ and $y_j^{} \in E$. At rate $\varrho_i^{}$, this block is, independently of all others, subdivided into two pieces, $A \cap C_i$ and $A \cap D_i$. Recall that $A \cap C_i$ inherits the label $y_j^{}$ from $A$ and the label of $A \cap D_i$ is $\emptyset$. Now, $j$ is the minimum of $A \cap C_i$, and $i$ is the minimum of $A \cap D_i$. Thus, $y_j^{}$ remains unchanged and $y_i^{}$ jumps to $\emptyset$. 
\end{enumerate}
To summarise:
\begin{lemma}\label{lem:restartedYs}
The stochastic processes $Y_i$, as defined above, are independent Markov processes. More precisely, $Y_{i_\bullet}^{}$ is a copy of $Y$, while the generator of $Y_i$ for $i \neq i_\bullet^{}$ is given by 
\begin{equation*}
\widetilde \cQ_i \varphi(y) \defeq \begin{cases}
\cQ \varphi (y) + r_i^{} \big ( \varphi(\emptyset) - \varphi(y) \big ) & \textnormal{ if } y \in E,\\
\varrho_i^{} (\varphi(\emptyset) - \varphi(\Delta) \big ) & \textnormal{ if } y = \Delta,
\end{cases}
\end{equation*}
and is defined for all measurable functions $\varphi$ on $E \cup \{\Delta\}$ such that $\varphi|_E^{} \in \cD(\cQ)$.
\qed
\end{lemma}
It is not hard to see that this generator defines a unique Feller process; for details, see the proof of Lemma~\ref{lem:existence} in the appendix.
\begin{example}
\begin{enumerate}
\item
When considering a deterministic flow as in Example~\ref{ex:semigroup}, the $Y_i^{}$ are precisely the \emph{initiation processes} introduced in \cite{SelrekDocumenta}. Namely, starting from $\Delta$, $Y_i^{}$ is started (or `initiated') at rate $\varrho_i^{}$, meaning that it moves from $\Delta$ to $0$. After that, it is reset at rate $r_i^{}$ and $Y_{i,t}$ is simply the time that has passed after the last resetting. 
\item
In the case of the selection-recombination equation, when $Y$ is the line-counting process of the binary Yule process as in Example~\ref{ex:pureselection}, then the $Y_i$ are the \emph{Yule processes with initiation and resetting}, short YPIR, also introduced in \cite{SelrekDocumenta}. Upon initiation, $Y_i^{}$ starts to undergo binary branching at rate $s$, and is at rate $r_i^{}$ reduced to a single line and thus `reset'.
\end{enumerate}
\end{example}

\begin{figure}[t]
\setlength{\unitlength}{1.3mm}
\begin{center}
  \begin{picture}(12,15)
    \put(-40,0){$(y_i^{})_{i \in S}^{}$}
    \put(-40,5){$\mathcal{A}$}
    \put(-40,10){$(v_A^{})_{A \in \cA}^{}$}
    \put(-13,10){$y_3^{}$}
    \put(-0.5,10){$y_5^{}$}
    \put(12,10){$y_6^{}$}
    \put(24.5,10){$y_9^{}$}
    \put(-23,5){$\big \{$}
    \put(-18,5){$1$}
    \put(-13,5){$2$}
    \put(-8,5){$3$}
    \put(-3,5){$4$}
    \put(2,5){$\underline{5}$}
    \put(7,5){$6$}
    \put(12,5){$7$}
    \put(17,5){$8$}
    \put(22,5){$9$}
    \put(27,5){$10$}
    \put(32,5){$\big \}$}
    \put(-18,0){$\Delta$}
    \put(-13,0){$\Delta$}
    \put(-8,0){$y_3^{}$}
    \put(-3,0){$\Delta$}
    \put(2,0){$y_5^{}$}
    \put(7,0){$y_6^{}$}
    \put(12,0){$\Delta$}
    \put(17,0){$\Delta$}
    \put(22,0){$y_9^{}$}
    \put(27,0){$\Delta$}
    \put(-19.5,5){$\{$}
    \put(-6.5,5){$\}$}
    \put(-4.5,5){$\{$}
    \put(3.5,5){$\}$}
    \put(5.5,5){$\{$}
    \put(18.5,5){$\}$}
    \put(20.5,5){$\{$}
    \put(29.75,5){$\}$}
  \end{picture} 
\end{center}
\caption{\label{fig:illustration_of_encoding}
Encoding  a labelled partition (upper two rows) by a labelling of the sites (bottom row). The selected site is underlined.}
\end{figure}

In analogy with \cite[Prop.~8.1]{SelrekDocumenta}, we can state the semigroup of $Y_i$ explicitly in terms of the semigroup $(P^Y_t)_{t \geqslant 0}^{}$ 
of $Y$. 
\begin{prop}
For $i \in S^\circ$, let $T_i$ and $S_i$ be exponentially distributed random variables with parameters $r_i^{}$ and $\varrho_i^{}$. Then, the semigroup of $Y_i$ is given by
\begin{equation*}
P^i_t (y,\cdot) = \begin{cases}
\EE \big [ \one_{T_i \geqslant t} Q^{}_t (y,\cdot) + \one_{T_i < t} Q^{}_{T_i} (\emptyset, \cdot) \big ] &\textnormal{ if $y \in E$}, \\
\EE \big [ \one_{S_i \geqslant t} \delta_\Delta(\cdot)+ \one_{S_i < t} Q^{}_{T_i \wedge t} (\emptyset, \cdot) \big ] &\textnormal{ if $y = \Delta$},
\end{cases}
\end{equation*}
where we recall that $Q = (Q_t)_{t \geqslant 0}$ is the transition semigroup of $Y$. In the degenerate case where $r_i^{} = \varrho_i^{} = 0$,
we let $T_i \equiv S_i \equiv \infty$ and the formula above reduces to
\begin{equation*}
P^i_t(y,\cdot) = \begin{cases}
P^Y_t (y,\cdot) & \textnormal{ if $y \in E$}, \\
 \delta_\Delta(\cdot) & \textnormal{ if $y = \Delta$}.
\end{cases}
\end{equation*}
\end{prop}
\begin{proof}
This is immediately obvious from the description of $Y_i$ in terms of its generator; again, the semigroup can be constructed explicitly, as in the proof of Lemma~\ref{lem:existence}.
\end{proof}

We conclude this abstract part of the paper by generalising the recursion for the solution of the selection-recombination equation to the $\psi$-recombination equation. 
For this, we fix a sequence $(i^{(k)})_{0 \leqslant k \leqslant n-1}^{}$, which is nondecreasing in the sense of the partial order `$\preccurlyeq$'
(see Fig.~\ref{fig:headsntails}), and set $\varrho^{(k)} \defeq \varrho_{i_k}^{}$. Note that this implies $i^{(0)} = i_\bullet^{}$ but otherwise, this sequence is in general not unique. An exception is the case $i_\bullet^{} \in \{1,n\}$, in which the partial order $\preccurlyeq$ is actually a \emph{total} order.

Now, we consider, for $0 \leqslant k \leqslant n-1$ the solutions of the $\psi$-recombination equation \emph{truncated at site} $k$, i.e. the solutions of the equations
\begin{equation}\label{truncatedeqs}
\dot \omega_t^{(k)} = \psi \big ( \omega_t^{(k)}  \big ) + \sum_{\ell = 1}^{k} \varrho^{(\ell)} \big ( \omega_{C^{(\ell)},t}^{(k)} \otimes \omega_{D^{(\ell)},t}^{(k)} - \omega_t^{(k)} \big ), \quad \omega^{(k)}_0 = \omega_0^{}.
\end{equation}
Note that this is consistent with our previous definition of $\omega^{(0)}_{}$.
\begin{theorem}\label{thm:recursion}
For \mbox{$1 \leqslant k \leqslant n-1$},  the solutions $\omega^{(k)}$ of the initial value problems in Eq.~\eqref{truncatedeqs} satisfy the following recursion.
\begin{equation*}
\omega_t^{(k)} = \ee^{-\varrho^{(k)} t} \omega_t^{(k-1)} + \omega_{C^{(k)},t}^{(k-1)} \otimes \int_0^t \varrho^{(k)} \ee^{-\varrho^{(k)} \tau} \omega_{D^{(k)},\tau}^{(k-1)} \dd \tau,
\end{equation*}
starting from the solution $\omega^{(0)}$ of 
\begin{equation*}
\dot \omega_t^{(0)} = \psi \big ( \omega_t^{(0)} \big ), \, \omega_0^{(0)} = \omega_0^{}.
\end{equation*}
\end{theorem}
\begin{proof}
We start with the case of two loci, i.e. $n=2$ and $i_\bullet^{} = 1$. Theorem~\ref{thm:duality} yields the stochastic representation
\begin{equation}\label{rawduality}
\omega_t^{(1)} = \EE \Big [ \bigotimes_{\sigma \in \Sigma^{(0)}_t} \Psi^{(0)}_{V_{\sigma,t}}(\omega_0^{})_\sigma^{} \, \Big | \, \Sigma_0^{(0)} = \{1,2\}, V_{\{1,2\},0} = 0 \Big ];
\end{equation}
recall that $\Psi^{(0)}$ is the flow associated with $\psi$. See also Example~\ref{ex:semigroup}. In this simple case, $\Sigma = (\Sigma^{(0)},V)$ evolves as follows. At rate $\varrho^{(1)}$, the sites $1$ and $2$ are split apart, and the label of site $1$ is reset to $0$, while the label of site $2$ remains unaffected. After that, the label of site $1$ is reset to $0$ at the same rate. Although time is positive, it is convenient to think of these events as the elements of a Poisson set $\Pi$, supported on the entire real line, with intensity $\varrho^{(1)} \dd t$; when $\Pi \cap (-\infty,t] \subseteq (-\infty,0]$, this means that no such event has occurred up to time $t$, and hence the block $\{1,2\}$ remains intact. Let $T \defeq t - \max_{\tau \in \Pi}^{}$. Then, $T$ is exponentially distributed with mean 
$1 / \varrho^{(1)}$ and 
\begin{equation*}
\Sigma_t^{} =
\begin{cases}
\big (\big \{ \{1\},\{2\} \big \},(t,T) \big ) & \text{ for } T < t, \\
(\{1,2\},t) & \text{ for } T \geqslant t.
\end{cases}
\end{equation*} 
Noting that $ \Psi^{(0)}_{V_{\sigma,t}}(\omega_0^{})_\sigma^{} = \omega^{(0)}_{\sigma,V_{\sigma,t}}$, Eq.~\eqref{rawduality} turns into
\begin{equation*}
\omega_t^{(1)} = \EE \big [ \one_{T \geqslant t} \, \omega_t^{(0)} + \one_{T < t} \, \omega^{(0)}_{\{1\},t} \otimes\omega^{(0)}_{\{2\},T} \big ] 
= \ee^{-\varrho^{(1)} t} \omega_t^{(0)} + \omega_{\{1\},t}^{(0)} \otimes \int_0^t \varrho^{(1)} \ee^{-\varrho^{(1)} \tau} \omega_{\{2\},\tau}^{(0)} \dd \tau,
\end{equation*}
which proves the claim for $n=2$ and $i_\bullet^{} = 1$. The general case follows from the following construction, which reduces the general case to the one we just proved. Let $n \geqslant 3$ and $1 \leqslant k \leqslant n-1$ be arbitary and set
  \begin{equation*}
\widetilde \psi (\nu) \defeq \psi (\nu) + \sum_{\ell = 1}^{k-1} \varrho^{(\ell)} \big (\nu_{C^{(\ell)}}^{} \otimes \nu_{D^{(\ell)}}^{} - \nu  \big ),
\end{equation*}
$
\widetilde S \defeq \{1,2\},
$
$
\widetilde X_1 \defeq X_{C^{(k)}}, \, \widetilde X_2 \defeq X_{D^{(k)}}, \, \widetilde \varrho \defeq \varrho^{(k)}, \, \widetilde X \defeq \widetilde X_1 \times \widetilde X_2,
$
$
\widetilde \omega \defeq \omega^{(k)},
$
and $\widetilde i_\bullet^{} = 1$.
With this, Eq.~\eqref{truncatedeqs} turns into
\begin{equation*}
\dot{\widetilde \omega}_t^{} = \widetilde \psi ( \widetilde \omega_t^{} ) + \widetilde \varrho \big ( \widetilde \omega_{\{1\},t}^{} \otimes 
\widetilde \omega_{\{2\},t}^{} - \widetilde \omega_t^{} \big ),
\end{equation*}
the $\widetilde \psi$-recombination equation for two sites. It remains to verify Assumption~\ref{ass:psiassumption} for $\widetilde \psi$.
First, we have for any $\nu, \nu' \in \cP(\widetilde X)$ with $\nu_{\{1\}}^{} = \nu'_{\{1\}}$ (that is, $\nu_{C^{(k)}}^{} = \nu_{C^{(k)}}^{}$ and hence 
$\nu_{\{i_\bullet^{}\}}^{} = \nu'_{\{i_\bullet^{} \}}$ in the original notation)
and $\alpha \in [0,1]$ that, 
by Assumption~\ref{ass:psiassumption}\ref{item:linearity} for $\psi$, $\psi \big (\alpha \nu + (1 - \alpha) \nu' \big ) = \alpha \psi(\nu) + (1 - \alpha) \psi(\nu')$. Furthermore, by the choice of the $i_k^{}$, we have for $1 \leqslant \ell \leqslant k-1$ that either $C^{(\ell)} \subseteq C^{(k)}$ (if $i_\ell^{}$ and $i_k^{}$ are on the same side of the selective site) or $D^{(\ell)} \subseteq C^{(k)}$ (if they are on different sides). The former implies that
$\nu_{C^{(\ell)}}^{} = \nu'_{C^{(\ell)}}$ and the latter that $\nu_{D^{(\ell)}}^{} = \nu'_{D^{(\ell)}}$. Thus, when evaluating
$\widetilde \psi \big (\alpha \nu + (1 - \alpha) \nu' \big )$, only one factor in the ocurring product measure is ever different from $\nu$, which implies that $\widetilde \psi \big (\alpha \nu + (1 - \alpha) \nu' \big ) = \alpha \widetilde \psi(\nu) + (1 - \alpha) \widetilde \psi(\nu)$. Moreover, the same reasoning together with assumption~\ref{ass:psiassumption}\ref{item:multiplicativity} on $\psi$ yields that $\widetilde \psi (\nu_{\{1\}} \otimes \nu'_{\{2\}}) = \widetilde \psi(\nu)_{\{1\}}^{} \otimes \nu_{\{2\}}^{}$. Therefore, $\widetilde \psi$ satisfies Assumption~\ref{ass:psiassumption}.
\end{proof}

\section{Application: The selection-mutation-recombination equation}
\label{sec:application}
We finish by applying our abstract results to the concrete example of the selection-mutation-recombination equation. In this section, $\psi$ is therefore of the form
\begin{equation*}
\psi = \psisel + \psimut,
\end{equation*}
where $\psi_{\textnormal{sel}}^{}$ and $\psi_{\textnormal{mut}}^{}$ describe the independent action of selection (at a single site) and mutation. We restrict (see Remark~\ref{rem:moregeneral}) ourselves to the case of binary sequences, so that $X = \{0,1\}^n$ throughout.
Explicitly, $\psi_{\textnormal{sel}}^{}$ is defined as in Example~\ref{ex:pureselection}, namely
\begin{equation}\label{psiseldef}
\psi_{\textnormal{sel}}^{} (\nu)(x) = s \big ( (1 - x_{i_\bullet^{}}^{}) \nu(x) - f(\nu) \nu(x) \big ), 
\end{equation}
where $f(\nu) \defeq \nu_{i_\bullet^{}}^{} (0)$ denotes the frequency of fit individuals. Intuitively, this means that an individual of type $x$ is considered to be \emph{fit} if $x_{i_\bullet^{}}^{} = 0$, and \emph{unfit} otherwise. Fit individuals reproduce at a higher rate $1+s$ (where $s > 0$) than unfit ones, which reproduce at rate $1$. Since we are working in a deterministic setting and assume that the population size remains constant, the neutral reproduction rate cancels out, and the net effect of the difference $s$ in reproduction rates is that any individual currently alive in the population is, at rate $s$ and regardless of its type, replaced by a randomly chosen fit individual. Defining the linear map $F : \RR^X \to \RR^X$ via 
\begin{equation*}
F \nu (x) \defeq (1 - x_{i_\bullet^{}}^{} ) \nu(x),
\end{equation*}
Eq.~\eqref{psiseldef} is equivalent to
\begin{equation*}
\psi_{\textnormal{sel}}^{} (\nu) = s(F\nu - f(\nu) \nu) = s (F - f(\nu) \id) \nu.
\end{equation*}
We recall the mutation term $\psi_{\textnormal{mut}}^{}$ given in Example~\ref{ex:mutation}.
\begin{equation*}
\psi_{\textnormal{mut}}^{} (\nu) \defeq \sum_{i \in S} u_i^{} \sum_{x \in X} \nu(x) \big ( m_{i,1} \delta_{(x_1^{},\ldots,x_{i-1}^{},1,x_{i+1}^{},\ldots,x_n^{})}^{}   +  m_{i,0} \delta_{(x_1^{},\ldots,x_{i-1}^{},0,x_{i+1}^{},\ldots,x_n^{})}^{} - \delta_x^{}  \big );
\end{equation*}
For general recombination, we will use Theorem~\ref{thm:duality} to obtain a stochastic representation of the solution and in the case of single-crossover, we iteratively construct an explicit solution formula via Theorem~\ref{thm:recursion}.

Unfortunately, Theorem~\ref{thm:duality} (and consequently, Theorem~\ref{thm:recursion}) does not apply immediately. As mutation can occur all over the genetic sequence, it is no longer true that a part of the genome that is statistically unlinked from $i_\bullet^{}$ remains unaffected by $\psi_{\textnormal{mut}}^{}$. Therefore, $\psi_{\textnormal{mut}}^{}$ (and consequently
$\psi$) violates Assumption~\ref{ass:psiassumption}\ref{item:multiplicativity}.

We overcome this slight inconvenience in two steps. First, we consider mutation at site $i_\bullet^{}$ only, in which case Assumption~\ref{ass:psiassumption}\ref{item:multiplicativity} will be satisfied. Secondly, we will take advantage of the linearity of the mutation term to add in mutation at the remaining sites.
So, we  decompose $\psimut$ as
\begin{equation*}
\psi_{\textnormal{mut}}^{} = \psi_{\textnormal{mut},\bullet} + \psi_{\textnormal{mut},S^\circ}^{},
\end{equation*}
where
\begin{equation*}
\psi_{\textnormal{mut},\bullet}(\nu) \defeq  u_{i_\bullet^{}}^{} \sum_{x \in X} \nu(x) \big ( m_{i_\bullet^{},1} \delta_{(x_1^{},\ldots,x_{i_\bullet^{}-1}^{},1,x_{i_\bullet^{}+1}^{},\ldots,x_n^{})}^{}   +  m_{i_\bullet^{},0} \delta_{(x_1^{},\ldots,x_{i_\bullet^{}-1}^{},0,x_{i_\bullet^{}+1}^{},\ldots,x_n^{})}^{} - \delta_x^{}  \big )
\end{equation*}
describes the mutation at $i_\bullet^{}$, while $\psi_{\textnormal{mut},S^\circ}^{}$ is given by
\begin{equation*}
\sum_{i \in S^\circ} u_i^{} \sum_{x \in X} \nu(x) \big ( m_{i,1} \delta_{(x_1^{},\ldots,x_{i-1}^{},1,x_{i+1}^{},\ldots,x_n^{})}^{}    +  m_{i,0} \delta_{(x_1^{},\ldots,x_{i-1}^{},0,x_{i+1}^{},\ldots,x_n^{})}^{} - \delta_x^{}  \big )
\end{equation*}
and describes mutation at the remaining sites.

We now show that $\psi_{\bullet}^{} \defeq \psi_{\textnormal{sel}}^{} + \psi_{\textnormal{mut},\bullet}^{}$ does indeed satisfy Assumption~\ref{ass:psiassumption}. 

\begin{lemma}\label{lem:psibulletisnice}
$\psi_{\bullet}^{}$ satisfies Assumption~\textnormal{\ref{ass:psiassumption}}.
\end{lemma}

\begin{proof}
For $\psi_{\textnormal{sel}}^{}$, this was already shown in~\cite[Eqs.~(15) and (19)]{SelrekTPB}. So, we now focus on verifying this for 
$\psi_{\textnormal{mut},\bullet}^{}$. Since $\psi_{\textnormal{mut},\bullet}^{}$ is linear, Assumption~\ref{ass:psiassumption}\ref{item:linearity} is satisfied. To show that $\psi_{\textnormal{mut},\bullet}^{}$ satisfies Assumption~\ref{ass:psiassumption}\ref{item:multiplicativity}, we assume now that $\nu = \nu_C^{} \otimes \nu_D^{}$ and that $i_\bullet^{} \in C$. For any $x \in X$, we let $x^1_C \defeq (x^1_i)_{i \in C}$ ($x^0_C \defeq (x^0_i)_{i \in C}$) where $x^1_i = x_i$ ($x^0_i = x_i$) for $i \in C \setminus \{i_\bullet^{}\}$ and $x^1_{i_\bullet^{}} = 1$ ($x^0_{i_\bullet^{}} = 0$) . 
Then, we have
\begin{equation*}
\begin{split}
\psi_{\textnormal{mut},\bullet}^{} (\nu) &= \psi_{\textnormal{mut},\bullet}^{} (\nu_C^{} \otimes \nu_D^{} ) \\
 &= u_{i_\bullet^{}}^{} \sum_{x \in X} \nu_C^{} \otimes \nu_D^{} (x)\big ( m_{i_\bullet^{},1} \delta_{x^1_C} \otimes \delta_{x_D^{}} + m_{i_\bullet^{},0} \delta_{x^0_C} \otimes \delta_{x_D^{}} - \delta_{x_C^{}} \otimes \delta_{x_D^{}} \big ) \\
&= u_{i_\bullet^{}}^{} \sum_{x_C^{} \in X_C^{}} \sum_{x_D^{} \in X_D^{}} \nu_C^{} (x_C^{}) \nu_D^{} (x_D^{})\Big ( \big ( m_{i_\bullet^{},1} \delta_{x^1_C}  + m_{i_\bullet^{},0} \delta_{x^0_C} - \delta_{x_C^{}}  \big ) \otimes \delta_{x_D^{}} \Big ) \\
&=  u_{i_\bullet^{}}^{} \Big ( \sum_{x_C^{} \in X_C^{}} \nu_C^{} (x_C^{})  \big ( m_{i_\bullet^{},1} \delta_{x^1_C}  + m_{i_\bullet^{},0} \delta_{x^0_C} - \delta_{x_C^{}}  \big ) \Big ) \otimes \Big ( \sum_{x_D^{} \in X_D^{}} \nu_D^{} (x_D^{}) \delta_{x_D^{}} \Big ) \\
&= \psi_{\textnormal{mut},\bullet}^{} (\nu)_C^{} \otimes \nu_D^{}.
\end{split}
\end{equation*}
By linearity, the claim for $\psi_\bullet^{} = \psi_{\textnormal{mut},\bullet}^{} + \psi_{\textnormal{sel}}^{}$ follows immediately.
\end{proof}

We denote the solution of Eq.~\eqref{psirecoeq} with $\psi_\bullet^{}$ in place of $\psi$ by $\omega^\bullet$. 
We obtain a stochastic representation of $\omega^\bullet$ via the gLPP, using the deterministic flow as in Example~\ref{ex:semigroup}. That is, each block $A$ is labelled by the time 
$V_{A,t}^{}$ that has passed since the last ($\cB$-)recombination event hit a block containing $A$ such that $A \subseteq B$ for some $B \in \cB^\circ$. Note that always $V_{A,t}^{} = t$
if $A$ contains $i_\bullet^{}$. In this setting, the role of the function $h$ in the definition of $H$ in Theorem~\ref{thm:duality} is played by the deterministic flow
$(t,\nu) \mapsto \Psi^{(0)\bullet}_{t} (\nu)$ associated with $\psi_\bullet^{}$. Thanks to Lemma~\ref{lem:psibulletisnice}, Theorem~\ref{thm:duality} is applicable and yields the following stochastic representation
\begin{equation}\label{bulletduality}
\omega^\bullet_t = \EE  \Big [ \bigotimes_{A \in \Sigma_t^{(0)}} \Psi^{(0)\bullet}_{V_{A,t}}(\nu)_A^{} \, \Big \vert \, (\Sigma_0^{(0)}, V_0) = (\{S\},0) \Big ].
\end{equation}
We interpret $V_{A,t}^{}$ as the amount of time that the sequence of ancestors of the sites in $A$ has been evolving under selection and mutation at site $i_\bullet^{}$; see Example~\ref{ex:semigroup}. We can illustrate the evolution of this version of the partitioning process via a random graph, drawn from right to left; see Fig.~\ref{fig:ancestry}. This generalises the \emph{ancestral initiation graph} (AIG) that was originally introduced in~\cite{SelrekTPB} in the context of single-crossover recombination.
\begin{enumerate}[label=(\arabic*)]
\item \label{item:singleline}
We start by tracing back the ancestral line of an individual sampled from the population at present, that is, at time $t$. This individual is represented by a single node, also called the \emph{root} of the graph. Its ancestral line, starting at the root, grows from right to left at unit speed. Because this line contributes the alleles at every site in $S$ (see step~\ref{item:propagation} below), we call this line \emph{ancestral to $S$}. 
\item \label{item:splitting}
Each line in the graph is hit, independently and at rate $\varrho_\cB^{}$ for each $\cB \in \bP(S)$, by a \emph{$\cB$-splitting}, which is marked by a square with label `$\cB$'. Then, to the left of the splitting, if the line in question used to be ancestral to $A$, the continuing line is labelled as ancestral to $A \cap B_\bullet$, where $B_\bullet$ is the head of $\cB$. For each $B \in \cB^\circ$, we attach a fresh line starting from below, again growing at unit speed, which is called ancestral to $A \cap B$. In the case that $A \cap B = \varnothing$, we call the line \emph{nonancestral}.
\item \label{item:sampling}
We stop the construction in steps~\ref{item:singleline} and~\ref{item:splitting} after time $t$. We interpret each left endpoint or \emph{leaf} as an ancestor of the sampled individual, contributing the alleles at a certain subset of loci. Thus, we associate a type to each leaf, sampled according to the age of the associated line. Here, the \emph{age} of the associated line is simply the time that has passed since it was (re)started or equivalently, its (horizontal) length. Moreover, if a line is ancestral to $A$, this means that the corresponding ancestor only contributed the sites in $A$, whence it is enough to sample the marginal type with respect to $A$. To summarise, if a line has age $\theta$ and is ancestral to $A$, we draw the type of its leaf from $\Psi^{\bullet (0)}_\theta (\omega_0^{})_A^{}$. In particular, we ignore the nonancestral lines in this step.
\item \label{item:propagation}
Finally, given the types of the leaves, we reconstruct the type of the root, i.e. the present-day individual, as follows. Propagate the types of the leaves from left to right along the associated lines. When a splitting event is encountered, the type at the line to the right is formed by combining the types at the lines to the left in the obvious way: If the lines to the left are ancestral to $A_1,\ldots,A_r$ and carry types $x_{A_1}^{} \in X_{A_1},\ldots,x_{A_r}^{} \in X_{A_r}$, then the line to the right carries the type $x_{A_1 \cup \ldots \cup A_r}^{}$ where $x_{A_1 \cup \ldots \cup A_r,i}^{} = x_{A_j,i}^{}$ for the unique $j$ such that $i \in A_j$.
\end{enumerate}
In this setting, the interpretation of equation~\eqref{bulletduality} is that, after averaging over all realisations of this random graph, the distribution of the type of the root is precisely $\omega_t^\bullet$. 

\begin{figure}[t]
\begin{tikzpicture}[xscale=0.7,yscale=1.4]
\draw[line width = 0.35mm, red] (13,0) -- (10.5,0);
\draw[line width = 0.35mm, blue] (13,-0.1) -- (10.5,-0.1);
\draw[line width = 0.35mm, purple] (13,0.1) -- (10.5,0.1);
\draw[line width = 0.35mm, purple] (10,0.1) -- (0,0.1);

\draw[line width = 0.35mm, red] (10.25,-0.1) -- (10.25,-1) -- (7,-1);
\draw[line width = 0.35mm, blue] (10.35,-0.1) -- (10.35,-2) -- (0,-2);
\draw[line width = 0.35mm, black!10] (6.5,-1) -- (0,-1);

\draw[line width = 0.35mm, red] (6.75,-1.25) -- (6.75,-1.5) -- (0,-1.5);

\draw[line width  = 0.35mm, black!10] (3.25,-0.25) -- (3.25, -0.5) -- (0, -0.5);

\draw[line width = 0.35mm, black!10] (4.25,-2.25) -- (4.25,-2.5) -- (0,-2.5);

\draw[line width = 0.35mm, fill = white] (9.85,0.20) rectangle (10.65,-0.20) node[pos=.5] {$\cA$};
\draw[line width = 0.35mm, fill = white] (6.35,-0.8) rectangle (7.15,-1.2) node[pos=.5] {$\cB$};
\draw[line width = 0.35mm, fill = white] (2.6,0.2) rectangle (3.4,-0.2) node[pos = .5] {$\cC$};
\draw[line width = 0.35mm, fill = white] (3.85,-1.8) rectangle (4.65,-2.2) node[pos = .5] {$\cB$};


\node[anchor = east] at (0,0.1) {$\scriptstyle{x_1^{} \sim \Psi_{\{1\},t}^{\bullet(0)}(\omega_0^{})^{}}$};
\node[anchor = east] at (0,-1.5) {$\scriptstyle{x_2^{} \sim \Psi_{\{2\},t_3^{}}^{\bullet(0)}(\omega_0^{})^{}}$};
\node[anchor = east] at (0,-2) {$\scriptstyle{x_3^{} \sim \Psi_{\{3\},t_4^{}}^{\bullet(0)}(\omega_0^{})^{}}$};


\node[anchor = west] at (13,0) {$\scriptstyle{(x_1^{},x_2^{},x_3^{})}$};


\draw[dashed] (0,-2.5) -- (0,0.5);
\draw[dashed]  (3.25,0.25) -- (3.25,0.5);
\draw[dashed] (4.25,-1.75) -- (4.25,0.5);
\draw[dashed] (6.75,-0.75) -- (6.75,0.5);
\draw[dashed] (10.25,0.25) -- (10.25,0.5);
\draw[dashed] (13,0.1) -- (13,0.5);

\node[anchor = south] at (0,0.5) {$\scriptstyle{0}$};
\node[anchor = south] at (3.25,0.5) {$\scriptstyle{t_1^{}}$};
\node[anchor = south] at (4.25,0.5) {$\scriptstyle{t_2^{}}$};
\node[anchor = south] at (6.75,0.5) {$\scriptstyle{t_3^{}}$};
\node[anchor = south] at (10.25,0.5) {$\scriptstyle{t_4^{}}$};
\node[anchor = south] at (13,0.5) {$\scriptstyle{t}$};
\end{tikzpicture}
\caption{\label{fig:ancestry}
An illustration of the ancestry of a single individual (right) living at time $t$, in the case $n=3$. The ancestral lines of sites $i_\bullet^{} = 1$,$2$ and $3$ are drawn in purple, red and blue, respectively, while the nonancestral lineages are drawn in grey. The partitions are chosen as follows: $\cA = \{\{1\},\{2\},\{3\}\}$, $\cB = \{\{1,3\},\{2\}\}$, $\cC = \{\{1\},\{2,3\}\}$. Note that as the (red) ancestral line of site $2$ is hit by a $\cB$-splitting, its age is reset since $2 \in \{2\} \in \cB^\circ$. In contrast, when the (blue) ancestral line of site $3$ is hit by a $\cB$-splitting, its age is not reset because $3 \in \{1,3\} = B_\bullet^{}$.}
\end{figure}
 
Of course, all of this is most useful if we can actually compute the semigroup $\Psi^{(0)\bullet}$! Fortunately, this is not difficult, because $\psi_\bullet^{}$ is \emph{almost}  linear. In fact, if we drop the assumption of constant population size, that is, if we dropped the normalisation terms  $-sf(\nu) \nu$ and $ - u_{i_\bullet^{}}^{}$ in the definition of $\psi_{\textnormal{sel}}^{}$ and $\psi_{\textnormal{mut},\bullet}^{}$ above, then $\psi_{\textnormal{sel}}^{}$ would just be $sF$, which is linear. Similarly, the mutation operator $\psi_\bullet^{}$ would just be $u_{i_\bullet^{}}^{} M_{i_\bullet^{}}$, where $M_{i_\bullet^{}}$ is a linear map, which is defined by its action on the canonical basis, given by the point masses.
\begin{equation}\label{actiononbasis}
\delta_{x_1^{},\ldots,x_n^{}}^{} \mapsto m_{i_\bullet^{},0}^{} \delta_{x_1^{},\ldots,x_{i_\bullet^{} - 1}^{},0,x_{i_\bullet^{} + 1}^{}}^{} 
+ m_{i_\bullet^{},1}^{} \delta_{x_1^{},\ldots,x_{i_\bullet^{} - 1}^{},1,x_{i_\bullet^{} + 1}^{}}^{}
\end{equation}
Thus, when ignoring the assumption of a constant population size, the semigroup would be given by the matrix exponential
\begin{equation*}
\ee^{t(sF + u_{i_\bullet^{}}^{} M_{i_\bullet^{}})}.
\end{equation*}
Indeed, we can recover the solution for the case of constant population size by normalising after the fact; this strategy is commonly known as `Thompson's trick'; see e.g.~\cite{ThompsonMcBride74}.
\begin{lemma}\label{lem:psibulletflow}
The semigroup $\Psi^{(0)\bullet}$ is given by
\begin{equation*}
\Psi_t^{(0)\bullet} = \frac{\ee^{t(sF + u_{i_\bullet^{}}^{} M_{i_\bullet})} \nu^{}}{\| \ee^{t(sF + u_{i_\bullet^{}}^{} M_{i_\bullet})} \nu^{} \|},
\end{equation*}
where $\|\mu\| = \sum_{x \in X} \mu(x)$ denotes the total mass of the nonnegative measure / vector $\mu$. Note that  the entries of $F$ and $M_{i_\bullet^{}}$ are nonnegative (the entries of $F$ are either $0$ or $1$ and $M_{i_\bullet}^{}$ is a Markov matrix).
\end{lemma}
\begin{proof}
This is a straightforward computation. For details, the reader is referred to~\cite{BaakeCorderoHummel21} for the 2-type case, and to~\cite[Section 5]{bb} for a more general discussion.
\end{proof}

\begin{figure}[t]
\begin{tikzpicture}[xscale=0.7,yscale=1.4]
\draw[line width = 0.35mm, red] (13,0) -- (10.5,0);
\draw[line width = 0.35mm, blue] (13,-0.1) -- (10.5,-0.1);
\draw[line width = 0.35mm, purple] (13,0.1) -- (10.5,0.1);
\draw[line width = 0.35mm, purple] (10,0.1) -- (0,0.1);

\draw[line width = 0.35mm, red] (10.25,-0.1) -- (10.25,-1) -- (7,-1);
\draw[line width = 0.35mm, blue] (10.35,-0.1) -- (10.35,-2) -- (0,-2);
\draw[line width = 0.35mm, black!10] (6.5,-1) -- (0,-1);

\draw[line width = 0.35mm, red] (6.75,-1.25) -- (6.75,-1.5) -- (0,-1.5);

\draw[line width  = 0.35mm, black!10] (3.25,-0.25) -- (3.25, -0.5) -- (0, -0.5);

\draw[line width = 0.35mm, black!10] (4.25,-2.25) -- (4.25,-2.5) -- (0,-2.5);

\draw[line width = 0.35mm, fill = white] (9.85,0.20) rectangle (10.65,-0.20) node[pos=.5] {$\cA$};
\draw[line width = 0.35mm, fill = white] (6.35,-0.8) rectangle (7.15,-1.20) node[pos=.5] {$\cB$};
\draw[line width = 0.35mm, fill = white] (2.85,0.2) rectangle (3.65,-0.2) node[pos = .5] {$\cC$};
\draw[line width = 0.35mm, fill = white] (3.85,-1.8) rectangle (4.65,-2.2) node[pos = .5] {$\cB$};


\node[anchor = east] at (0,0.1) {$\scriptstyle{x_1^{} \sim \Psi_{\{1\},t}^{\bullet(0)}(\omega_0^{})^{}}$};
\node[anchor = east] at (0,-1.5) {$\scriptstyle{x_2^{} \sim \Psi_{\{2\},t_3^{}}^{\bullet(0)}(\omega_0^{})^{}}$};
\node[anchor = east] at (0,-2) {$\scriptstyle{x_3^{} \sim \Psi_{\{3\},t_4^{}}^{\bullet(0)}(\omega_0^{})^{}}$};


\node[anchor = west] at (13,0) {$\scriptstyle{(x_1^{},0,1)}$};

\draw (5,-1.5) pic[red, line width = 0.35mm] {cross = 4pt};
\draw[line width = 0.35mm, red] (8,-1) ellipse (4pt and 2pt);
\draw (3,-2) pic[blue, line width = 0.35mm] {cross = 4pt};
\draw(11,-0.1) pic[blue,line width = 0.35mm] {cross = 4pt};

\node[anchor = north west] at (2.5,-2) {$\scriptstyle{x_3^{} \to 1}$};
\node[anchor = north west] at (5,-1.5) {$\scriptstyle{x_2^{} \to 1}$};
\node[anchor = north west] at (8,-1) {$\scriptstyle{1 \to 0}$};
\node[anchor = north west] at (11,-0.1) {$\scriptstyle{1 \to 1}$};


\draw[dashed] (0,-2.5) -- (0,0.5);
\draw[dashed]  (3.25,0.25) -- (3.25,0.5);
\draw[dashed] (4.25,-1.75) -- (4.25,0.5);
\draw[dashed] (6.75,-0.75) -- (6.75,0.5);
\draw[dashed] (10.25,0.25) -- (10.25,0.5);
\draw[dashed] (13,0.1) -- (13,0.5);

\node[anchor = south] at (0,0.5) {$\scriptstyle{0}$};
\node[anchor = south] at (3.25,0.5) {$\scriptstyle{t_1^{}}$};
\node[anchor = south] at (4.25,0.5) {$\scriptstyle{t_2^{}}$};
\node[anchor = south] at (6.75,0.5) {$\scriptstyle{t_3^{}}$};
\node[anchor = south] at (10.25,0.5) {$\scriptstyle{t_4^{}}$};
\node[anchor = south] at (13,0.5) {$\scriptstyle{t}$};
\end{tikzpicture}
\caption{\label{fig:ancestrywithmutations}
The ancestry of a single individual (right), with added mutations on the ancestral lines of sites $2$ and $3$. As before, the ancestral line of site $1 = i_\bullet^{},2$ and $3$ are drawn in purple, red and blue, respectively, while the nonancestral lineages are drawn in grey. The elements of the Poisson sets $A_i$ and $B_i$ are drawn as circles and crosses. Whenever a (blue, red) cross is enountered, the allele (at site $2,3$) changes to $1$,  whenever a (blue, red) cross is enountered, the allele (at site $2,3$) changes to $0$. Note that there are no mutation events on the purple line, because mutations at site $1 = i_\bullet^{}$ are already accounted for by the flow $\Psi^{\bullet(0)}$.
}
\end{figure}
In the case of single-crossover, Theorem~\ref{thm:recursion} then allows us to recursively compute the solution.
\begin{coro}\label{coro:smrsolution}
Let $(i_k^{})_{0 \leqslant k \leqslant n-1}^{}$, $\varrho^{(k)}$, $C^{(k)}$ and $D^{(k)}$  be as in Section~\textnormal{\ref{sec:singlecrossover}}  and let $\omega^{\bullet(k)}$, for each $1 \leqslant k \leqslant n-1$, be the solution of the initial value problem
\begin{equation*}
\dot \omega^{\bullet(k)}_t = \psi_\bullet^{} (\omega^{\bullet(k)}_t) + \sum_{\ell = 1}^{k} \varrho^{(\ell)} \big ( \omega^{\bullet(k)}_{C^{(\ell)},t} 
\otimes \omega^{\bullet(k)}_{D^{(\ell)},t} - \omega_t^{\bullet(k)} \big ), \,  \omega^{\bullet(k)}_0 = \omega_0^{}.
\end{equation*}
In addition, let $\omega^{\bullet (0)}_t \defeq \Psi^\bullet_t (\omega_0^{})$ \textnormal{(}see Lemma~\textnormal{\ref{lem:psibulletflow}}\textnormal{)}.
Then, the family $\big (\omega^{\bullet(k)}\big)_{1 \leqslant k \leqslant n-1}$ satisfies the recursion 
\begin{equation*}
\omega_t^{\bullet(k)} = \ee^{-\varrho^{(k)} t} \omega_t^{\bullet(k-1)} + \omega_{C^{(k)},t}^{\bullet(k-1)} \otimes \int_0^t \varrho^{(k)} \ee^{-\varrho^{(k)} \tau} \omega_{D^{(k)},\tau}^{\bullet(k-1)} \dd \tau.
\end{equation*}
\end{coro}
We still have to deal with mutation at the remaining sites. In analogy to Eq.~\eqref{actiononbasis}, we define for each $i \in S^\circ$ a linear map $M_i$ via
\begin{equation*}
\delta_{x_1^{},\ldots,x_n^{}}^{} \mapsto m_{i,0}^{} \delta_{x_1^{},\ldots,x_{i_\bullet^{} - 1}^{},0,x_{i_\bullet^{} + 1}^{}}^{} 
+ m_{i,1}^{} \delta_{x_1^{},\ldots,x_{i_\bullet^{} - 1}^{},1,x_{i_\bullet^{} + 1}^{}}^{}
\end{equation*}

This is not difficult because $\psi_\bullet^{}$ and $M_i$ for $i \in S^\circ$ act on different sites, and therefore commute.

\begin{theorem}\label{thm:finalsolution}
The solution $\omega_t^{}$ of Eq.~\eqref{psirecoeq} with $\psi = \psi_{\textnormal{sel}}^{} + \psi_{\textnormal{mut}}^{}$ and initial condition 
$\omega_0^{}$ is given by 
\begin{equation*}
\omega_t^{} = \exp \Big (t \sum_{i \in S^\circ}  u_i^{} (M_i  - \id)   \Big ) \omega^\bullet_t,
\end{equation*}
where $\omega_t^\bullet = \omega^{\bullet (n-1)}_t$ is the solution of the initial value problem 
\begin{equation*}
\dot \omega_t^{\bullet} = \psi_\bullet ( \omega_t^{\bullet}) + \sum_{\ell = 1}^{n-1} \big ( \omega_{C^{(k)},t}^{\bullet} \otimes \omega_{D^{(k)},t}^{\bullet} - \omega_t^{\bullet}   \big ), \, \, \omega_0^\bullet = \omega_0^{}.
\end{equation*}
\hfill 
\end{theorem}
\begin{proof}
First, note that $M_i$ and $M_j$ commute for $i \neq j$. Additionally, for $i \in S^\circ$ and $\nu \in \cP(X)$,
\begin{equation*}
F M_i \nu = \sum_{\substack{x \in X \\ x_{i_\bullet^{}}^{} = 0}} \nu(x) \big ( M_{i,0} \delta_{(x_1^{},\ldots,x_{i-1}^{},0,x_{i+1}^{},\ldots,x_n^{})} +  
M_{i,1} \delta_{(x_1^{},\ldots,x_{i-1}^{},1,x_{i+1}^{},\ldots,x_n^{})}  \big ) = M_i F \nu,
\end{equation*}
so $F M_i = M_i F$ for all $i \in S^\circ$. Thus, abbreviating $\cE_t \defeq \exp \Big (t \sum_{i \in S^\circ}  u_i^{} (M_i  - \id)   \Big )$, we have
$\psi_\bullet^{} \circ \cE_t = \cE_t \circ \psi_\bullet$. Moreover, for all $\nu \in \cP(X)$ and $\{A,B\} \in \bP(S)$, we have
$\cE_t (\nu_A^{} \otimes \nu_B^{} ) = \cE_t(\nu)_A \otimes \cE_t(\nu)_B$. Thus, 
\begin{equation*}
\begin{split}
\frac{\dd}{\dd t} \cE_t \omega^\bullet_t 
&= \sum_{i \in S^\circ} u_i (M_i - \id) \cE_t \omega_t^\bullet + \cE_t \Big ( \psi_\bullet^{} (\omega_t^\bullet)
+ \sum_{i \in S^\circ} \varrho_i^{} \big ( \omega_{C_i,t}^\bullet \otimes \omega_{D_i,t}^\bullet -\omega_t^\bullet   \big )    \Big ) \\
&=  \sum_{i \in S^\circ} u_i (M_i - \id) \cE_t \omega_t^\bullet  + \psi_\bullet  ( \cE_t \omega_t^\bullet  ) + \sum_{i \in S^\circ} \varrho_i^{}
\big ( (\cE_t \omega_t^\bullet)_{C_i} \otimes (\cE_t \omega_t^\bullet)_{D_i} - \cE_t \omega_t^\bullet    \big ) \\
&= \psi (\cE_t \omega_t^\bullet) +  \sum_{i \in S^\circ} \varrho_i^{}
\big ( (\cE_t \omega_t^\bullet)_{C_i} \otimes (\cE_t \omega_t^\bullet)_{D_i} - \cE_t \omega_t^\bullet    \big ).
\end{split}
\end{equation*}
\end{proof}

This last result admits the following stochastic interpretation. Note that for any $i$, \mbox{$u_i^{}(M_i - \id)$} is the transpose of the generator matrix of the \emph{mutation process at site $i$}, that is, the random walk on $X$ that jumps from any $x \in X$ at rate $u_i^{} m_{i,0}$ ($u_i^{} m_{i,1}$) to
$(x_1^{},\ldots,x_{i-1}^{},0,x_{i+1}^{},\ldots,x_n^{})$ ($(x_1^{},\ldots,x_{i-1}^{},1,x_{i+1}^{},\ldots,x_n^{})$). 
Thus, we can interpret $\omega_t^{}$ in Theorem~\ref{thm:finalsolution} as the distribution at time $t$ of the mutation processes, acting independently at every site in $S^\circ$, with initial distribution $\omega_t^\bullet$ (which would be the type distribution when neglecting mutations at the sites in $S^\circ$). 

Graphically, this can be realised by decorating the ancestral lines (as illustrated in
 Fig.~\ref{fig:ancestry}) with independent Poisson point processes $\Pi_{0,i}$ and $\Pi_{1,i}$, with intensities $u_i^{} m_{i,0} \dd t$ and $u_i^{} m_{i,1} \dd t$. When propagating the types
 through the graph, the allele of an individual at site $i$ is changed to $0$ when it encounters an element of $\Pi_{0,i}$ and is changed to $1$ if it
 encounters an element of the Poisson set $\Pi_{1,i}$. This is illustrated in Fig.~\ref{fig:ancestrywithmutations}.

\begin{remark}\label{rem:moregeneral}
In this section, we have restricted ourselves to the case $X = \{0,1\}^n$ of binary sequences. Inspecting the arguments shows that replacing the sets of alleles at all sites except the selected site by arbitrary $X_i$ poses no real challenge and mostly increases notational complexity (for infinite $X_i$ one would, of course, need to impose an appropriate integrability condition on the mutation kernel). Even more generally, $\psi_{\textnormal{mut},\circ}^{}$ could be replaced by an arbitrary linear term (not necessarily with constant coefficients) which commutes with 
$\psi_{\textnormal{sel}}^{} + \psi_{\textnormal{mut},\bullet}^{}$.

 It is less trivial to allow more general alphabets at the selected site, because there are, in general, no closed formulas for the solution of the selection equation for more than two types. Although whenever such a solution can be found, Theorems 5.3 and 5.4 would still allow to add an arbitrary number of neutral sites with arbitrary sets of alleles. In particular, the stochastic representation discussed thereafter holds also in the multitype situation.
\end{remark}

\appendix
\section{Proof of Lemma~\ref{lem:existence}}
\label{app:proof}
\begin{proof}
By our assumption, $E$ is a Polish space whose topology is induced by some metric $\tilde d$; without loss of generality, we may assume that $\tilde d (y,z) < 1$ for all $y,z \in E$. We can equip $F$ with the metric $d$ given by
\begin{equation*}
d \big ( (\cA,v), (\cB,w)  \big ) \defeq
\begin{cases}
1, & \textnormal{if } \cA \neq \cB, \\
\frac{1}{n} \sum_{A \in \cA} \tilde d(v_A^{},w_A^{}) & \textnormal{if } \cA = \cB.
\end{cases}
\end{equation*}
It is clear that $d$ induces on each subset of the form $\{ \cA \} \times E^\cA$ the product topology of $E$. Hence, all of these subspaces are Polish, and so is their finite union, $F$.

We now construct the gLPP starting from $(\cA,v)$. Let $(\widetilde \tau_i^{})^{}_{i \in \NN}$ be a sequence of i.i.d. exponentially distributed random variables with mean $1 / \varrho$, where $\varrho \defeq \sum_{\cA \in \bP(S)} \varrho_\cA^{}$. Additionally, we define independent copies $Y^{(k)}_A =  (Y^{(k)}_{A,t})_{t \geqslant 0}^{}$ of $Y$ for each $A \subseteq S$ and $k \in \NN \defeq \{1,2,3,\ldots\}$. When $A \in \cA$ and $k=1$, $Y_A^{(k)}$ is started from $v_A^{}$, otherwise from $\emptyset$.

We will now inductively define $\tau_k^{}$, $\Sigma|_{[0,\tau_k^{})}$ and $V^{(k)} = (V^{(k)}_t)^{}_{t \geqslant 0}$  for each $k \in \NN$. Here, $\tau_k^{}$ is the time of the $k$-th splitting/resetting event and $V^{(k)}$ is the process of labels if splitting/resetting events are suspended after the $k$-th one.

We start by letting $V^{(1)} \defeq (Y^{(1)}_{A,t})^{}_{A \in \cA, t \geqslant 0}$,  $\tau_1^{} \defeq \widetilde \tau_1^{} / |\cA|$. 
Then, we set $\Sigma_t \defeq (\cA,V^{(1)}_t)$ for all $t \in [0,\tau_1^{})$.

Given  $\tau_k^{}$, $\Sigma|_{[0,\tau_k^{})}$ and $V^{(k)}$ for some $k \geqslant 1$, we choose a random partition $\cB$ with probability $\varrho_\cB^{} / \varrho$. Also, choose a block $A$ uniformly from $\Sigma_{\tau_k^{}-}^{}$. We let
\begin{equation*}
\tau_{k+1}^{} \defeq \tau_k^{} + \frac{\widetilde \tau_{k+1}^{}}{|\Sigma_{\tau_k^{}-}|}.
\end{equation*}
For $t \in [\tau_k^{},\tau_{k+1}^{})$, we let 
\begin{equation*}
\Sigma_t^{(0)} \defeq \big ( \Sigma_{\tau_k^{}-} \setminus \{A\} \big ) \cup \cB|_A.
\end{equation*}
Then, for all $t \in [\tau_k^{}, \infty)$, we set
\begin{equation*}
V^{(k+1)}_{A',t} \defeq 
\begin{cases}
Y^{(k+1)}_{A',t - \tau_k^{}} & \textnormal{for } A' \in \big ( \Sigma_{\tau_k^{}-} \setminus \{A\} \big ) \cup \cB^\circ|_{A}, \\
V^{(k)}_{A',t} & \textnormal{for } A' = A \cap B_\bullet^{}.
\end{cases}
\end{equation*}

Finally, for $t \in [\tau_k^{},\tau_{k+1}^{})$, we define 
$
V_t^{} \defeq V_t^{(k+1)}.
$
Thus, $\Sigma_t$ is now defined for all $t \in [\tau_k^{}, \tau_{k+1}^{})$. In particular, $\Sigma_{\tau_k^{}} = \Sigma_{\tau_k^{} -} \boxwedge_A \cB$.

Note that $\tau_{k+1}^{} - \tau_k^{}$ is, for all $k$, bounded from below by $\widetilde \tau_{k+1}^{} / n$. Because $\sum_{j = 1}^\infty \widetilde \tau_j^{} = \infty$ almost surely, this inductive procedure almost surely defines $\Sigma$ on all of $[0,\infty)$. 
The Markov property of $\Sigma$ is due to the memorylessness of the exponential distribution. 

Next, we verify the Feller property. Fix $t \geqslant 0$. To show that the map 
\begin{equation*}
(\cA,v) \mapsto \EE \big [ f(\Sigma_t) \mid \Sigma_0 = (\cA,v)  \big ]
\end{equation*}
is in $\cC_0(F)$ (the set of real-valued continuous functions on $F$ that vanish at infinity), it suffices to prove that the map 
\begin{equation*}
v  \mapsto \EE \big  [ f(\Sigma_t) \mid \Sigma_0 = (\cA,v)  \big ]
\end{equation*}
is in $\cC_0(E^\cA)$ for any fixed $\cA \in \bP(S)$. We fix $t \in \RR_{\geqslant 0 }$ and decompose
\begin{equation*}
\Sigma^{(0)}_t = \Sigma^{(o)}_t \cup \Sigma_t^{(r)},
\end{equation*}
where $\Sigma^{(o)}_t$ contains the blocks who still carry the ``original'' labels  $Y^{(1)}$.

In addition, given any $\cB \in \bP(S)$, $v \in E^\cB$ and $\cC \subseteq \cB$, we abbreviate $ (v_C^{})_{C \in \cC}^{}$ as $v_\cC^{}$. Similarly, we define $Q^\cC$ as the semigroup of an independent collection of copies of $Y$, associated to each block of $\cC$, acting on $\cC_0(E^\cC)$. Clearly, this semigroup inherits the Feller property from $Q$, as can be seen from Fubini's theorem.
Decompose
\begin{equation} \label{dsagfac}
\begin{split}
\EE \big [ f(\Sigma_t) &\mid \Sigma_0 = (\cA,v)  \big ] \\
&=
\sum_{\substack{(\cC,\cD) \in \bP(S)^2 \\ \cC \cap \cD = \varnothing, \cC \cup \cD = \Sigma_t} } \EE \big [ f(\Sigma_t) \mid \Sigma_0 = (\cA,v), \Sigma^{(o)}_t = \cC, \Sigma^{(r)}_t = \cD  \big ]
\PP \big ( \Sigma^{(o)}_t = \cC, \Sigma^{(r)}_t = \cD \big ).
\end{split}
\end{equation} 
Now, notice that  the conditional expectation
\begin{equation*}
\EE \big [ f(\Sigma_t) \mid \Sigma_0 = (\cA,v), \Sigma^{(o)}_t = \cC, \Sigma^{(r)}_t = \cD  \big ]
\end{equation*}
does not depend on $v_\cD^{}$. Moreover, by dominated convergence, it is continuous in $v_{\cC}^{}$. Thus,
\begin{equation*}
\EE \big [ f(\Sigma_t) \mid \Sigma_0 = (\cA,v), \Sigma^{(o)}_t = \cC, \Sigma^{(r)}_t = \cD  \big ]
= Q_t^{\cC} \EE \big [ f(\Sigma_t) \mid \Sigma_0 = (\cA,v_{\cD}^{},\cdot), \Sigma^{(o)}_t = \cC, \Sigma^{(r)}_t = \cD  \big ] (v_{\cC}^{}),
\end{equation*}
Since we assumed that $Y$ is a Feller process and by Fubinis's theorem, the semigroup $Q_t^{\cC}$ is Feller. Hence, the map 
\begin{equation*}
v \mapsto Q_t^{\cC} \EE \big [ f(\Sigma_t) \mid \Sigma_0 = (\cA,v^{(r)},\cdot), \Sigma^{(o)}_t = \cC, \Sigma^{(r)}_t = \cD  \big ] (v_\cC^{})
\end{equation*}
is in $\cC_0(E^\cA)$, and, by Eq.~\eqref{dsagfac}, so is the map
\begin{equation*}
v \mapsto \EE \big [ f(\Sigma_t) \mid \Sigma_0 = (\cA,v)  \big ].
\end{equation*}

Next, we need to prove the strong continuity of the map $t \mapsto P_t f$ for any fixed $f \in \cC_0(F)$. Due to the semigroup property, it suffices to show that
$\lim_{t \to 0} \| P_t f - f \|_\infty = 0$. Let $(\cA,v) \in F$ be fixed. Recall that $\tau_1^{}$ is the time of the first resetting/splitting event. By the triangle inequality,
\begin{equation*}
\begin{split}
|P_t f (\cA,v) - f(\cA,v)| &= | \EE \big [ f(\Sigma_t) \mid \Sigma_0 = (\cA,v)  \big ] - f(\cA,v) | \\
&=
|\EE \big [ f(\Sigma_t) \mid \Sigma_0 = (\cA,v), \tau_1^{} > t    \big ] - f (\cA,v) | + \cO(t) \\
&=
|Q_t^\cA f(\cA,v) - f(\cA,v)| + \cO(t) \\
& \leqslant
\max_{\cA \in \bP(S)} \|Q_t^\cA f - f \|_\infty + \cO(t),
\end{split}
\end{equation*}
where the last expression does not depend on $(\cA,v)$ and converges to $0$ as $t \to 0$, by the strong continuity of $Q^\cA$ for each $\cA \in \bP(S)$. In the first step, we used that 
$\PP(\tau_1^{} \leqslant t) = 1 - \ee^{-\varrho |\cA| t} \leqslant 1 - \ee^{-n \varrho t} = \cO(t)$.

It remains to show that $\cL$ is indeed the generator of $\Sigma$ (or rather, its associated semigroup). We start by decomposing according to the time of the first recombination event,
\begin{equation*}
\begin{split}
\EE \big [ & f(\Sigma_h) - f(\cA,v) \mid \Sigma_0 = (\cA,v)   \big ]  \\
&= \EE \big [ f(\Sigma_h) - f(\cA,v) \mid \Sigma_0 = (\cA,v),  \tau_1^{} \leqslant h    \big ] \PP(\tau_1^{} \leqslant h) \\
&\quad +  \EE \big [ f(\Sigma_h) - f(\cA,v) \mid \Sigma_0 = (\cA,v),  \tau_1^{} > h    \big ] \PP(\tau_1^{} > h).
\end{split}
\end{equation*}
We start by investigating the first term. Since $\tau_1^{}$ is exponentially distributed, we have
\mbox{$\PP(\tau_1^{} \leqslant h) = 1 - \ee^{-\varrho |\cA| h} = \varrho |\cA| h + \cO(h^2)$}. Thus, it is enough to consider the preceding expectation up to $o(1)$.
Conditional on $\tau_1^{} \leqslant h$, the probability that also $\tau_2^{} \leqslant h$ is $\cO(h)$. Thus, 
\begin{equation*}
\begin{split}
\EE \big [ & f(\Sigma_h) - f(\cA,v) \mid \Sigma_0 = (\cA,v),  \tau_1^{} \leqslant h    \big ] \PP(\tau_1^{} \leqslant h)  \\
& = \EE \big [  \big ( f(\Sigma_h) - f(\cA,v)  \big ) \one_{\tau_1^{} \leqslant h}^{} \mid \Sigma_0 = (\cA,v),  \tau_1^{} \leqslant h, \tau_2^{} > h \big ] + o(h) \\
& = \int_0^h \varrho |\cA| \ee^{-\varrho |\cA| \eta} \sum_{A \in \cA} \frac{1}{|\cA|} \sum_{\cB \in \bP(S)} \frac{\varrho_\cB^{}}{\varrho} \Big (Q_{h - \eta}^{\cB^\circ|_A} Q_h^{B_\bullet^{}} Q_h^{\cA|_{A^c}} 
f \big ((\cA,v) \boxwedge_A \cB \big ) - f(\cA,v) \Big ) \dd \eta + o(h) \\
&= \int_0^h \varrho |\cA| \ee^{-\varrho |\cA| \eta} \dd \eta \sum_{A \in \cA} \frac{1}{|\cA|} \sum_{\cB \in \bP(S)} \frac{\varrho_\cB^{}}{\varrho}
\Big ( f \big ((\cA,v) \boxwedge_A \cB \big ) - f(\cA,v) \Big ) + o(h) \\
&= h \sum_{A \in \cA} \sum_{\cB \in \bP(S)} \varrho_\cB^{} \Big ( f \big ( (\cA,v) \boxwedge_A \cB  \big ) - f(\cA,v)        \Big ) + o(h),
\end{split}
\end{equation*} 
where we made use of the fact that $f \in \widetilde{D} (\cL)$ in the third step. Next, we deal with the second term.  
\begin{equation*}
\begin{split}
\EE \big [ & f(\Sigma_h) - f(\cA,v) \mid \Sigma_0 = (\cA,v),  \tau_1^{} > h    \big ] \PP(\tau_1^{} > h) \\
&= \big ( Q^\cA_h f(\cA,v) - f(\cA,v) \big ) \ee^{-\varrho |\cA| h}  \\
&= \Big ( h \sum_{j = 1}^{|\cA|} \cQ_j f (\cA,\cdot) (v)  + o(h) \Big ) \big (1 - \varrho |\cA| h + \cO(h^2) \big ) \\
&=  h \sum_{j = 1}^{|\cA|} \cQ_j f (\cA,\cdot) (v)  + o(h).
\end{split}
\end{equation*}
Putting this together proves the claim. Finally, note that, for Feller processes, the semigroup is uniquely determined by the semigroup; see~\cite[Prop.~2.10]{EundK}

\end{proof}

\section*{Acknowledgements}
It is a pleasure to thank Ellen and Michael Baake for a careful reading of and helpful suggestions to improve this manuscript. This project was funded by the Deutsche Forschungsgemeinschaft (DFG, German Research Foundation) – Project-ID 317210226 – SFB 1283.

\bigskip

\bigskip

\end{document}